\documentclass[final]{siamltex}
\usepackage{latexsym,amsmath,amssymb}
\usepackage{amsfonts}

\usepackage{comment}

\usepackage{graphicx}
\usepackage{epsfig}
\usepackage{tabularx}
\usepackage{multirow}
\usepackage{color}
\usepackage{chngpage}

\newtheorem{example}{Example}

\newtheorem{remark}{Remark}

\newcommand{\dtn}{\Delta t_n}
\newcommand{\dt}{\Delta t}
\newcommand{\nexp}{N_{\text{exp}}}

\allowdisplaybreaks
\title{Fast evaluation of the Caputo fractional derivative and its
  applications to fractional diffusion equations}

\author{Shidong Jiang
\thanks{Department of Mathematical Sciences, New Jersey Institute of Technology,
  Newark, NJ 07102, USA ({\tt shidong.jiang@njit.edu}).
  The research of this author was supported by NSF under grant DMS-1418918.}
\and
Jiwei Zhang
\thanks{Beijing Computational Science Research Center, Beijing 100094, China
  ({\tt jwzhang@csrc.ac.cn}).
  The research of this author was supported by the National Natural Science Foundation of China under grant
91430216.}
\and
Qian Zhang
\thanks{Beijing Computational Science Research Center, Beijing, China ({\tt qianzhang@csrc.ac.cn}).}
\and
Zhimin Zhang
\thanks{Beijing Computational Science Research Center, Beijing, China ({\tt zmzhang@csrc.ac.cn}),
and Department of Mathematics, Wayne State University, Detroit, MI 48202, USA ({\tt zzhang@math.wayne.edu}).
  The research of this author was supported in part by the National Natural Science Foundation of China under
  grants 11471031, 91430216, and by NSF under grant DMS-1419040.}
}

\begin{document}

\maketitle

\begin{abstract}
  We present an efficient algorithm for the evaluation of the Caputo
  fractional derivative $_0^C\!D_t^\alpha f(t)$ of order $\alpha\in (0,1)$, which
  can be expressed as a convolution of $f'(t)$ with the kernel
  $t^{-\alpha}$. The algorithm is based on an efficient 
  sum-of-exponentials approximation for the kernel $t^{-1-\alpha}$
  on the interval $[\dt, T]$ with a uniform absolute error
  $\varepsilon$, where the number of exponentials $\nexp$ needed is of the order
  $O\left(\log\frac{1}{\varepsilon}\left(
    \log\log\frac{1}{\varepsilon}+\log\frac{T}{\dt}\right)
   +\log\frac{1}{\dt}\left(
     \log\log\frac{1}{\varepsilon}+\log\frac{1}{\dt}\right)
    \right)$.
  As compared with the direct method, the resulting algorithm reduces the storage
  requirement from $O(N_T)$ to $O(\nexp)$ and the overall computational
  cost from $O(N_T^2)$ to $O(N_T\nexp)$ with $N_T$ the total number of time
  steps. Furthermore, when the fast evaluation scheme of the Caputo derivative is
  applied to solve the fractional
  diffusion equations, the resulting algorithm requires
  only $O(N_S\nexp)$ storage and $O(N_SN_T\nexp)$ work with $N_S$
  the total number of points in space; whereas
  the direct methods require $O(N_SN_T$) storage
  and $O(N_SN_T^2)$ work. The complexity
  of both algorithms is nearly optimal since $\nexp$ is of the order
  $O(\log N_T)$ for $T\gg 1$ or $O(\log^2N_T)$ for $T\approx 1$ for fixed
  accuracy $\varepsilon$. We also present a detailed stability and error analysis
  of the new scheme for solving linear fractional diffusion equations. The
  performance of the new algorithm is illustrated via several numerical examples.
  Finally, the algorithm can be parallelized in a straightforward manner.
\end{abstract}



\pagestyle{myheadings}
\thispagestyle{plain}
\markboth{S. Jiang, J. Zhang, Q. Zhang, and Z. Zhang}
{Fast Evaluation of the Caputo Fractional Derivative}

\section{Introduction} \label{sec:intro}

Over the last few decades the fractional calculus has received much attention
of both physical scientists and mathematicians since they can faithfully
capture the dynamics of physical process in many applied sciences including
biology, ecology, and control system. The anomalous diffusion, also referred to as
the non-Gaussian process, has been observed and validated in many phenomena with
accurate physical measurement \cite{KST06,KH10,MK00,OS74,P99}. The mathematical
and numerical analysis of the factional calculus became a subject of intensive
investigations. In literature, there are several definitions of fractional
time derivatives including the Riemann-Liouville (RL) fractional derivative,
the Gr\"unwald-Letnikov (GL) fractional derivative, and the Caputo fractional derivative (see, for example, \cite{P99}
for details). It is easy to see that the GL fractional derivative is equivalent
to the RL fractional derivative. Both require fractional-type initial values,
whose physical interpretation is not quite clear. On the other hand, the Caputo
fractional derivative takes the integer-order differential equations as the
initial value, and the Caputo fractional derivative of a constant is zero,
just as one would expect for the usual derivative.

In this paper, we are concerned with the evaluation of the Caputo fractional
derivative, which is defined by the formula
\begin{equation}\label{defc}
_{0}^{C}\!D_{t}^{\alpha}u(t)
  =\frac{1}{\Gamma(m-\alpha)}\int^{t}_{0}\frac{ u^{(m)} (\tau)}
  {(t-\tau)^{\alpha +1-m }}d\tau,\quad m-1<\alpha<m, \quad m\in \mathbb{Z}.
\end{equation}
One of the popular schemes of discretizing the Caputo fractional
derivative is the so-called $L1$ approximation
\cite{Gao12,Gao13,GSJCP11,LHJCP05,LDCMA11,LXJCP07,LXJCP11,Sun:06},
which is simply based on the piecewise
linear interpolation of $u$ on each subinterval. For $0<\alpha<1$, the order of
accuracy of the $L1$ approximation is $2-\alpha$. There are
also high-order discretization scheme by using piecewise high-order
polynomial interpolation of $u$
\cite{CXJCP13,GS14,LCYJCP11,OAMC06}.
These methods require the storage of all previous function values
$u(0), u(\Delta t), \cdots, u(n\Delta t)$ and $O(n)$ flops at the $n$th step. Thus
it requires on average $O(N_T)$ storage and
the total computational cost is $O(N_T^2)$ with $N_T$ the total number of time
steps, which forms a bottleneck for long time simulations, especially when one
tries to solve a fractional partial differential equation.

Here we present an efficient scheme for the evaluation of the
Caputo fractional derivative $_{0}^{C}\!D_{t}^{\alpha}u(t)$ for $0<\alpha<1$.
We first split the convolution integral in \eqref{defc} into two parts -
a local part containing the integral from $t-\dt$ to $t$,
and a history part containing the integral from $0$ to $t-\dt$.
The local part is approximated using the standard $L1$ approximation.
For the history part, integration by part leads to a convolution integral
of $u$ with the kernel $t^{-1-\alpha}$. We show that
$t^{-1-\alpha}$ ($0<\alpha<1$) admits an efficient sum-of-exponentials
approximation on the interval $[\dt, T]$ with a uniform absolute error
$\varepsilon$ and the number of exponentials needed is of the order

\begin{equation}\label{nexp}
\nexp= O\left(\log\frac{1}{\varepsilon}\left(
    \log\log\frac{1}{\varepsilon}+\log\frac{T}{\dt}\right)
   +\log\frac{1}{\dt}\left(
     \log\log\frac{1}{\varepsilon}+\log\frac{1}{\dt}\right)
     \right).
\end{equation}
That is, for fixed precision $\varepsilon$, we have $\nexp=O(\log N_T)$ for
$T\gg 1$ or $\nexp=O(\log^2N_T)$ for $T\approx 1$ assuming that
$N_T=\frac{T}{\dt}$. The approximation can be used to accelerate the
evaluation of the convolution via the standard recurrence relation.
The resulting algorithm has nearly optimal complexity - $O(N_T\nexp)$ work
and $O(\nexp)$ storage.

We would like to remark here that sum-of-exponentials approximations have
been applied to speed up the evaluation of the convolution integrals in
many applications. In fact, they have been used to accelerate the evaluation
of the heat potentials in \cite{greengard1,greengard2,jiang5}, and
the evaluation of the exact nonreflecting boundary conditions for the wave,
Schr\"{o}dinger, and heat equations in \cite{alpert,alpert2,jiang1,jiang2,jiang3,Zheng}.
There are also many other efforts to accelerate the evaluation of fractional
derivatives; see, for example, \cite{LS13,WM12,PS12,WW10,WT12} and references therein.

We then apply the fast evaluation scheme of the Caputo fractional derivative
to study the fractional diffusion equations (both linear and nonlinear).
We demonstrate that it is straightforward to incorporate the fast evaluation
scheme of the Caputo fractional derivative into the existing standard finite
difference schemes for solving the fractional diffusion equations.
The resulting algorithm for solving the fractional PDEs is both efficient
and stable. The computational cost of the new algorithm is $O(N_S N_T\nexp)$
as compared with $O(N_SN_T^2)$ for direct methods and the
storage requirement is only $O(N_S\nexp)$ as compared with
$O(N_S N_T)$ for direct methods, since one needs to store the solution
in the whole computational spatial domain at all times. Furthermore, we have
carried out a rigorous and detailed analysis to prove that our scheme
is unconditionally stable with respect to arbitrary step sizes.
With these two properties, our scheme provides an efficient and reliable
tool for long time large scale simulation of fractional PDEs.

The paper is organized as follows. In Section 2, we describe the fast algorithm
for the evaluation of the Caputo fractional derivative and provide rigorous
error analysis of our discretization scheme. In Section 3, we apply our fast
algorithm to solve the linear fractional diffusion PDEs
and present the stability and error analysis for the overall scheme. In Section 4,
we study the nonlinear fractional diffusion PDEs and demonstrate that our
fast algorithm has the same order of convergence as the direct method
in this case. Finally, we conclude our paper with a brief discussion
on the extension and generalization of our scheme.

\section{Fast Evaluation of the Caputo Fractional Derivative}
In this section, we consider the fast evaluation of the Caputo fractional
derivative for $0<\alpha<1$. Suppose that we would like to evaluate the
Caputo fractional derivative on the interval $[0,T]$ over a set of
grid points $\Omega_t := \{t_n,\ n=0,1,\cdots,N_T\}$, with $t_0=0$, $t_{N_T}=T$,
and $\dtn = t_n -t_{n-1}$. We will simply denote $u(t_n)$ by $u^n$.

We first split the convolution integral in \eqref{defc}
into a sum of local part and history part, that
is,
\begin{align}\label{2.0.1}
  {^{C}_{0}}\!{D}_t^{\alpha}u^{n} &=
  \frac{1}{\Gamma(1-\alpha)} \int_0^{t_n} \frac{u'(s)ds}{(t_n-s)^\alpha} \nonumber\\
&= \frac{1}{\Gamma(1-\alpha)} \int_{t_{n-1}}^{t_n} \frac{u'(s)ds}{(t_n-s)^\alpha}
  + \frac{1}{\Gamma(1-\alpha)} \int_0^{t_{n-1}} \frac{u'(s)ds}{(t_n-s)^\alpha}
  \nonumber\\
&:= C_l(t_n)+C_h(t_n),
\end{align}
where the last equality defines the local part and the history part,
respectively. For the local part, we apply the standard $L1$ approximation,
which approximates $u(s)$ on $[t_{n-1}, t_n]$ by a linear polynomial
(with $u^{n-1}$ and $u^n$ as the interpolation nodes) or $u'(s)$ by a constant
$\frac{u(t_n) - u(t_{n-1})}{\dtn}$. We have
\begin{align}\label{2.0.2}
C_l(t_n)
\approx \frac{u(t_n) - u(t_{n-1})}{\dtn \Gamma(1-\alpha)}
\int_{t_{n-1}}^{t_n} \frac{1}{(t_n-s)^\alpha}ds = \frac{u(t_n) - u(t_{n-1})}
    {\dtn^\alpha \Gamma(2-\alpha)} .
\end{align}
For the history part, we apply the integration by part to eliminate
$u'(s)$ and have 
\begin{align}\label{2.0.3}
  C_h(t_n)&=  \frac{1}{\Gamma(1-\alpha)} \int_0^{t_{n-1}}\frac{u'(s)ds}{(t_n-s)^\alpha}\nonumber\\
  &= \frac{1}{\Gamma(1-\alpha)} \left[\frac{u(t_{n-1})}{\dt_n^\alpha}
    - \frac{u(t_0)}{t_{n}^\alpha} - \alpha \int_0^{t_{n-1}}
    \frac{u(s)ds}{(t_n-s)^{1+\alpha}} \right].
\end{align}
\subsection{Efficient Sum-of-exponentials Approximation for the Power Function}
We now show that the convolution kernel $t^{-1-\alpha}$ ($0<\alpha<1$)
can be approximated via a sum-of-exponentials approximation efficiently
on the interval $[\Delta t, T]$ with the absolute error $\varepsilon$. That is,
there exist positive real numbers $s_i$ and $w_i$ ($i=1,\cdots,\nexp$) such that
for $0<\alpha<1$,
\begin{equation}\label{2.2.1}
  \left|\frac{1}{t^{1+\alpha}}-\sum_{i=1}^{\nexp} \omega_i e^{-s_i t}\right| \leq
  \varepsilon, \quad t\in[\Delta t, T],
\end{equation}
where $\nexp$ is given by \eqref{nexp}. Our proof is constructive and the error
bound is explicit. We start from the following integral representation
of the power function.

\begin{lemma}\label{soe1}
  For any $\beta>0$,
  \begin{equation}\label{2.1.1}
    \frac{1}{t^\beta} = \frac{1}{\Gamma(\beta)}\int_0^\infty e^{-ts}s^{\beta-1}ds.
  \end{equation}
\end{lemma}
\begin{proof}
  This follows from a change of variable $x=ts$ and the integral definition of the $\Gamma$ function
  \cite{handbook}.
\end{proof}

\eqref{2.1.1} can be viewed as a representation of $t^{-\beta}$ using an infinitely many (continuous)
exponentials. In order to obtain an efficient sum-of-exponentials approximation, we first truncated
the integral to a finite interval, then subdivide the finite interval into a set of dyadic intervals and
discretize the integral on each dyadic interval with proper quadratures.

We now assume $1<\beta<2$, which is the case we are concerned with in this paper.
\begin{lemma}\label{soe2}
  For $t\geq \delta >0$,
  \begin{equation}\label{2.1.2}
    \left|\frac{1}{\Gamma(\beta)}\int_p^\infty e^{-ts}s^{\beta-1}ds\right|\leq e^{-\delta p}
    2^{\beta-1}\left(\frac{p^\beta}{\Gamma(\beta)}+\frac{1}{\delta^\beta}\right).
  \end{equation}
\end{lemma}
\begin{proof}
\begin{equation}\label{2.1.3}
\begin{aligned}
\left|\frac{1}{\Gamma(\beta)}\int_p^\infty e^{-ts}s^{\beta-1}ds\right|
&= \left|\frac{1}{\Gamma(\beta)}e^{-tp}\int_0^\infty e^{-tx}(x+p)^{\beta-1}dx\right|\\
&\leq e^{-\delta p}\frac{1}{\Gamma(\beta)}\left(\int_0^p (2p)^{\beta-1}dx
+\int_p^\infty e^{-\delta x}(2x)^{\beta-1}dx\right)\\
&\leq e^{-\delta p}\frac{2^{\beta-1}}{\Gamma(\beta)}\left(p^\beta
+\int_0^\infty e^{-\delta x}x^{\beta-1}dx\right)
\\
&\leq e^{-\delta p}2^{\beta-1}
\left(\frac{p^\beta}{\Gamma(\beta)}
+\frac{1}{\delta^\beta}\right).
\\
\end{aligned}
\end{equation}
\end{proof}

\begin{remark}
  The truncation error can be made arbitrarily small for fixed $\delta$ by choosing sufficiently large $p$.
  Usually we have $\delta<1$ and if one would like to bound the truncation error by $\varepsilon<1/e$,
  then $\delta p>1$ or $p>1/\delta$, and
  \begin{equation}\label{2.1.4}
    \begin{aligned}
e^{-\delta p}2^{\beta-1}\left(\frac{p^\beta}{\Gamma(\beta)}+\frac{1}{\delta^\beta}\right)
&<e^{-\delta p}2^{\beta-1}\left(\frac{p^\beta}{\Gamma(\beta)}+p^\beta\right)<5e^{-\delta p}p^2.
    \end{aligned}
    \end{equation}
  Thus, $p=O\left(\log\left(\frac{1}{\varepsilon\delta}\right)/\delta\right)$ is sufficient
  to bound the truncation error by $\varepsilon$.
\end{remark}

We now proceed to discuss the discretization error for the integral on the interval $[0,p]$.
Similar to \cite{jiang2}, we will analyze the discretization error
on each dyadic interval using the Gauss-Legendre quadrature.
\begin{lemma}\label{soe3}
  Consider a dyadic interval $[a, b] = [2^j, 2^{j+1}]$ and let $s_1,\cdots, s_n$ and
  $w_1,\cdots,w_n$ be the nodes and weights for $n$-point Gauss-Legendre
  quadrature on the interval. Then for $\beta\in (1,2)$ and $n>1$,
  \begin{equation}\label{2.1.5}
\left|\int_a^b e^{-ts}s^{\beta-1}ds-\sum_{k=1}^n w_ks_k^{\beta-1}e^{-s_k t}\right|
    <2^{\beta - \frac{3}{2}}\pi  a^\beta \left(\frac{e^{1/e}}{4}\right)^{2n}.
 \end{equation}

\end{lemma}

\begin{proof}
  For any interval $[a, b]$, the standard estimate for $n$-point Gauss-Legendre
  quadrature \cite{handbook} yields
  \begin{equation}\label{2.1.6}
\left|\int_a^b e^{-ts}s^{\beta-1}ds-\sum_{k=1}^n w_ks_k^{\beta-1}e^{-s_k t}\right|
\le \frac{(b-a)^{2n+1}}{2n+1}\frac{(n!)^4}{[(2n)!]^3}
\max_{s\in(a,b)}\left|D^{2n}_s (e^{-st}s^{\beta-1})\right|,
 \end{equation}
  where $D_s$ denotes the derivative with respect to $s$.\\
  Applying Stirling's approximation \cite{handbook}
  \begin{equation}\label{2.1.7}
\sqrt{2\pi}n^{n+1/2}e^{-n}<n!<2\sqrt{\pi}n^{n+1/2}e^{-n},
  \end{equation}
  we obtain
  \begin{equation}\label{2.1.8}
    \frac{(n!)^4}{[(2n)!]^3}<2\sqrt{\pi}\left(\frac{e}{8}\right)^{2n}
    \frac{\sqrt{n}}{n^{2n}}.
  \end{equation}
  Observe now $|(\beta-1)\cdots(\beta-k)| \leq k!/4$ for $k>1$, and thus 
  \begin{equation}\label{2.1.9}
\begin{aligned}
\left|D_s^k s^{\beta-1}\right| &= s^{\beta-1}  \leq \frac{\sqrt{\pi}}{2}\sqrt{2n}(2n)^ke^{-k}s^{\beta-k-1}, \quad \text{for } k =0,\\
\left|D_s^k s^{\beta-1}\right| &= (\beta-1) s^{\beta-k-1} \leq  \frac{\sqrt{\pi}}{2}\sqrt{2n}(2n)^ke^{-1}s^{\beta-k-1} , \quad \text{for } k=1,\\
\left|D_s^k s^{\beta-1}\right| &= \left|(\beta-1)\cdots(\beta-k) s^{\beta-k-1}\right|\leq \frac{k!}{4} s^{\beta-k-1}\\
&\leq \frac{\sqrt{\pi}}{2}k^{k+1/2}e^{-k}s^{\beta-k-1}\leq \frac{\sqrt{\pi}}{2}\sqrt{2n}(2n)^ke^{-k}s^{\beta-k-1}, \quad \text{for } k>1.\\
\end{aligned}
\end{equation}
  We also have
  \begin{equation}\label{2.1.10}
    D_s^{2n-k} e^{-st} =(-t)^{2n-k}e^{-st}.
  \end{equation}
 Combining \eqref{2.1.9}, \eqref{2.1.10} with the Leibniz rule, we obtain
  \begin{equation}\label{2.1.11}
    \begin{aligned}
      \left|D^{2n}_s (e^{-st}s^{\beta-1})\right|
      &= \left|\sum_{k=0}^{2n}   \left(\begin{array}{c}2n\\k\\ \end{array}\right)
      \left(D_s^{2n-k}e^{-st}\right) \left(D_s^k s^{\beta-1}\right)\right|\\
      &\leq \sqrt{2n}s^{\beta-1}e^{-st}
      \sum_{k=0}^{2n}   \left(\begin{array}{c}2n\\k\\ \end{array}\right)
      t^{2n-k} (2n)^ke^{-k}s^{-k}\\
      &=\sqrt{2n}s^{\beta-1}e^{-st} \left(t+\frac{2n}{e s}\right)^{2n}.
    \end{aligned}
  \end{equation}
 Combining \eqref{2.1.6}, \eqref{2.1.8}, \eqref{2.1.11} and $b-a=a$, we have
\begin{align}
  \left|\int_a^b e^{-ts}s^{\beta-1}ds-\sum_{k=1}^n w_ks_k^{\beta-1}e^{-s_k t}\right|
  &\leq \frac{\sqrt{2}\pi}{2}(b-a)s^{\beta-1}e^{-st}\left(\frac{e(b-a)t}{8n}+\frac{b-a}{4s}\right)^{2n}\nonumber\\
  &\leq 2^{\beta - \frac{3}{2}}\pi a^\beta e^{-at}\left(\frac{eat}{8n}+\frac{1}{4}\right)^{2n}.\label{2.1.12}
\end{align}
And \eqref{2.1.5} follows from the fact
\begin{equation}\label{2.1.13}
  \max_{x>0}e^{-x}\left(\frac{ex}{8n}+\frac{1}{4}\right)^{2n}
  = \left(\frac{e^{1/e}}{4}\right)^{2n}, \quad n\geq 2.
\end{equation}
\end{proof}

We now consider the end interval $[0,a]$.
\begin{lemma}\label{soe4}
  Let $s_1,\cdots, s_n$ and
  $w_1,\cdots,w_n$ ($n\geq 2$) be the nodes and weights for $n$-point Gauss-Jacobi
  quadrature with the weight function $s^{\beta-1}$
  on the interval. Then for $0<t<T$, $\beta\in (1,2)$ and $n>1$,
  \begin{equation}\label{2.1.14}
\left|\int_0^a e^{-ts}s^{\beta-1}ds-\sum_{k=1}^n w_ke^{-s_k t}\right|
    <2\sqrt{\pi}a^\beta n^{3/2} \left(\frac{e}{8}\right)^{2n}
 \left(\frac{aT}{n}\right)^{2n}.
 \end{equation}

\end{lemma}
\begin{proof}
The standard estimate for $n$-point Gauss-Jacobi
quadrature \cite{handbook} yields
\begin{equation}\label{2.1.15}
\left|\int_0^a e^{-ts}s^{\beta-1}ds-\sum_{k=1}^n w_ke^{-s_k t}\right|
\le \frac{a^{2n+\beta}}{2n+\beta}\frac{(n!)^2[\Gamma(n+\beta)]^2}
{(2n)![\Gamma(2n+\beta)]^2}
\max_{s\in (0,a)}\left|D^{2n}_s e^{-st}\right|.
\end{equation}
For $n>1$, we have $\Gamma(n+\beta)<\Gamma(n+2)=(n+1)!$,
$\Gamma(2n+\beta)>\Gamma(2n+1)=(2n)!$, $2n+\beta>2n+1$.
Thus,
\begin{equation}
\begin{aligned}\label{2.1.16}
\left|\int_0^a e^{-ts}s^{\beta-1}ds-\sum_{k=1}^n w_ke^{-s_k t}\right|
&\leq\frac{a^{2n+\beta}}{2n+1}\frac{(n!)^2[(n+1)!]^2}
          {[(2n)!]^3}t^{2n}e^{-st}\\
&\leq 2\sqrt{\pi}a^\beta n^{3/2} \left(\frac{e}{8}\right)^{2n}
 \left(\frac{aT}{n}\right)^{2n}.
\end{aligned}
\end{equation}
\end{proof}

We are now in a position to combine the last three lemmas to give an
efficient sum-of-exponentials approximation for $t^{-\beta}$
on $[\delta, T]$ for $\beta\in (1,2)$. The proof is straightforward.
\begin{theorem}\label{soethm}
Let $0<\delta\leq t \leq T$ ($\delta\leq 1$ and $T\geq 1$),
let $\varepsilon > 0$ be the desired
precision, let $n_{o}=O(\log\frac{1}{\varepsilon}) $,
let $M=O(\log T)$,
and let $N=O(\log\log\frac{1}{\varepsilon}+\log\frac{1}{\delta})$.
Furthermore, let ${s_{o,1}, \dots, s_{o,n_{o}}}$ and
${w_{o,1}, \dots, w_{o,n_{o}}}$
be the nodes and weights for the $n_{o}$-point Gauss-Jacobi quadrature
on the interval $[0,2^{M}]$, let ${s_{j,1}, \dots, s_{j,n_s}}$ and
${w_{j,1}, \dots, w_{j,n_s}}$
be the nodes and weights for $n_s$-point Gauss-Legendre quadrature
on small intervals $[2^j,2^{j+1}]$, $j=M,\cdots,-1$, where
$n_s= O\left(\log\frac{1}{\varepsilon}\right)$,
and let ${s_{j,1}, \dots, s_{j,n_l}}$ and
${w_{j,1}, \dots, w_{j,n_sl}}$
be the nodes and weights for $n_l$-point Gauss-Legendre quadrature
on large intervals $[2^j,2^{j+1}]$, $j=0,\cdots,N$, where
$n_l= O\left(\log\frac{1}{\varepsilon}+\log\frac{1}{\delta}\right)$.
Then for $t\in [\delta,T]$ and $\beta\in (1,2)$,
\begin{equation}
  \left|\frac{\Gamma(\beta)}{t^\beta} -
  \left(\sum_{k=1}^{n_{o}} e^{-s_{o,k}t}w_{o,k}
  +\sum_{j=M}^{-1}\sum_{k=1}^{n_s} e^{-s_{j,k}t}s_{j,k}^{\beta-1}w_{j,k}
  +\sum_{j=0}^{N}\sum_{k=1}^{n_l} e^{-s_{j,k}t}s_{j,k}^{\beta-1}w_{j,k}
  \right)
  \right|
\leq \varepsilon.
\end{equation}
\end{theorem}
\begin{remark}
The important fact which emerges from this theorem is that the total
number of exponentials needed to approximate
$t^{-\beta}$ for $0<\dt\leq t\leq T$ with an absolute
error $\varepsilon$ is given by the formula \eqref{nexp}.
\end{remark}
\begin{remark}
Efficient sum-of-exponentials approximation for the power function
$t^{-\beta}$ ($\beta>0$) has been studied in detail both anaytically
and algorithmically in a sequence of papers \cite{beylkin1,beylkin2,beylkin3}.
In \cite{beylkin3}, it has been shown that for any $\beta>0$ the power function
$t^{-\beta}$ admits an efficient sum-of-exponentials approximation on
the interval $[\delta, 1]$ with a {\it relative} error $\varepsilon$, and
the number of terms needed is $p=O\left(\log\frac{1}{\varepsilon}
\left(\log\frac{1}{\varepsilon}+\log\frac{1}{\delta}\right)\right)$.
The proof in \cite{beylkin3} is constructive and relies on the truncated
trapezoidal rule to discretize an integral from $-\infty$ to $\infty$. Along
the lines of \cite{beylkin3}, it is straightforward to show that the number
of exponentials needed will be
$O\left(\left(\log\frac{1}{\varepsilon}+\log\frac{T}{\delta}\right)^2\right)$
if one wants to bound the {\it absolute} error on the interval $[\delta, T]$.

In \cite{Li10SISC}, it has been shown that for any $0<\beta<1$ the power function
$t^{-\beta}$ admits an efficient sum-of-exponentials approximation on
the interval $[\delta, \infty]$ with an {\it absolute} error $\varepsilon$, and
the number of terms needed is
$O\left(\left(\log\frac{1}{\varepsilon}+\log\frac{1}{\delta}\right)^2\right)$.
The proof in \cite{Li10SISC} is also constructive, although it relies on the
adaptive Gaussian quadrature and utilizes asymptotic error formula for the Gauss
quadrature.

The important difference between our result \eqref{nexp} and the above theoretical
results is that there is only $O\left(\log\frac{1}{\varepsilon}
\log\log\frac{1}{\varepsilon}\right)$ term in \eqref{nexp}, while
$O\left(\log^2\frac{1}{\varepsilon}\right)$ term appears in both \cite{beylkin3}
and \cite{Li10SISC}.
\end{remark}
\begin{remark}\label{reduction}
The resulting number of exponentials following the construction in Theorem
\ref{soethm} is unnecessarily large. One may apply modified Prony's method
in \cite{beylkin2} to reduce the number of exponentials for nodes on the
interval $(0,1)$, while standard model reduction method in \cite{jiang6} can
be applied to reduce the number of
exponentials for nodes on the interval $[1,p]$.
\end{remark}\\
Table \ref{sumnum} lists the actual number of exponentials needed to
approximate $t^{-1-\alpha}$ with various precisions $\varepsilon$
and $N_T=T\slash \Delta t$ after applying the reduction algorithms in Remark
\ref{reduction}. We observe that the number of exponentials needed
is very modest even for high accuracy approximations. Indeed, one needs less than
$80$ terms in order to march one million steps with $9$-digit accuracy.
\begin{table}[!ht]
\centering
\caption{Number of exponentials needed to approximate $t^{-1-\alpha}$.
The second row indicates the ratio $T\slash \Delta t$ with fixed $\Delta t=10^{-3}$.
The first column stands for the absolute error $\varepsilon$.}  \label{sumnum}
\begin{adjustwidth}{1.8cm}{0cm}
\resizebox{!}{1.4cm}
{
\begin{tabular}{lccccccccccc} \hline \multicolumn{1}{c}{}&\multicolumn{1}{c}{\multirow{2}{0.2cm}{}} &\multicolumn{4}{c}{$\alpha=0.2$} &\multicolumn{1}{c}{\multirow{2}{0.2cm}{}}  &\multicolumn{4}{c}{$\alpha=0.5$} \\
\cline{1-1}\cline{3-6}  \cline{8-11} \multicolumn{1}{c}{{$\varepsilon$}$\bigg\backslash${$\frac{T}{\Delta t}$}}&\multicolumn{1}{c}{} &\multicolumn{1}{c}{}  {$10^3$} &{$10^4$}&{$10^5$}&{$10^6$} &\multicolumn{1}{c}{}  & {$10^3$} &{$10^4$}&{$10^5$}&{$10^6$} \\ \hline
    $10^{-3}$  &&27&31&35&38&      &28&30&35&38&\\
    $10^{-6}$  &&37&42&47&52&      &42&47&47&51&\\
   $10^{-9}$   &&47&54&59&64&      &49&55&64&72&\\
  \hline
\end{tabular}
}
\end{adjustwidth}
\end{table}

\subsection{Fast Evaluation of the History Part}
We replace the convolution kernel $t^{-1-\alpha}$ by
its sum-of-exponentials approxiamtion in \eqref{2.2.1} to
approximate the history part defined in \eqref{2.0.3} as follows:
\begin{align}
C_h(t_n)& \approx  \frac{1}{\Gamma(1-\alpha)} \left[\frac{u(t_{n-1})}{\Delta t^\alpha} - \frac{u(t_0)}{t_{n}^\alpha} - \alpha \sum_{i=1}^{\nexp} \omega_i \int_0^{t_{n-1}}e^{-(t_n-\tau)s_i} u(\tau)d\tau \right] \nonumber\\
& =  \frac{1}{\Gamma(1-\alpha)} \left[\frac{u(t_{n-1})}{\Delta t^\alpha} - \frac{u(t_0)}{t_{n}^\alpha} - \alpha
  \sum_{i=1}^{\nexp} \omega_i U_{\text{hist,i}}(t_{n})  \right].
\end{align}
To evaluate $U_{\text{hist,i}}(t_n) $ for $n=1,2,\cdots,N_T$, we observe
the following simple recurrence relation:
\begin{equation}\label{eq3.3}
U_{\text{hist,i}}(t_n) = e^{-s_{i}\Delta t}U_{\text{hist,i}}(t_{n-1}) +
\int_{t_{n-2}}^{t_{n-1}}e^{-s_{i}(t_n-\tau)}
u(\tau)d\tau.
\end{equation}
At each time step, we only need $O(1)$ work to compute $U_{\text{hist,i}}(t_n) $ since $U_{\text{hist,i}}(t_{n-1}) $ is known at that point. Thus, the
total work is reduced from $O(N_T^2)$ to $O(N_T\nexp)$, and the total
memory requirement is reduced from $O(N_T)$ to $O(\nexp)$.

One may compute the integral on the right hand side of \eqref{eq3.3} by interpolating $u$ via a linear
function and then evaluating the resulting approximation analytically. We have
\begin{equation}\label{eq3.4}
\begin{aligned}
\int_{t_{n-2}}^{t_{n-1}}e^{-s_{i}(t_n-\tau)}
u(\tau)d\tau&\approx\frac{e^{-s_i \Delta t}}{s_i^2\Delta t}
\left[(e^{-s_i \Delta t}-1+s_i\Delta t)u^{n-1}\right.\\
&\quad+\left.(1-e^{-s_i \Delta t}-e^{-s_i \Delta t}s_i\Delta t)u^{n-2}\right].
\end{aligned}
\end{equation}
We note that the weights in front of $u^{n-1}$ and $u^{n-2}$ in \eqref{eq3.4}
are subject to significant cancellation error when $s_i\Delta t$ is small. In that
case, we can compute the weights by a Taylor expansion of exponentials
with a small number of terms.

\begin{remark}
Another popular fast method for computing the
convolution with exponential functions is to solve the equivalent initial value
problem for an ordinary differential equation. We would like to point
out that in our case this may force one to choose a very small time step
$\Delta t$ for the overall scheme. This is because
$s_i$ ($i=1,2,\cdots,\nexp$) usually varies in orders of different
magnitudes and the resulting ODE system will be very stiff.
Thus we prefer to evaluate
the convolution via the simple recurrence relation \eqref{eq3.3}.
\end{remark}

\subsection{Error Analysis}

It is straightforward to verify that our scheme of evaluating the Caputo
fractional derivative is equivalent to the following formula

\begin{align}
  {^{C}_{0}}\!{D}_t^{\alpha}u^{n} & \approx
  \frac{u(t_n) - u(t_{n-1})}{(1-\alpha)\Delta t^\alpha \Gamma(1-\alpha)}+
  \frac{1}{\Gamma(1-\alpha)} \left[\frac{u(t_{n-1})}{\Delta t^\alpha}
    - \frac{u(t_0)}{t_{n}^\alpha} - \alpha \sum_{i=1}^{\nexp}
    \omega_i U_{\text{hist,i}}(t_n)  \right]\nonumber\\
  &=\frac{\Delta t^{-\alpha}}{\Gamma(1-\alpha)}\left(\frac{u^n}{1-\alpha}
  -(\frac{\alpha}{1-\alpha}+a_0)u^{n-1}\right.\nonumber\\
  &\left.-\sum_{l=1}^{n-2}(a_{n-l-1}+b_{n-l-2})u^l
  -(b_{n-2}+\frac{1}{n^{\alpha}})u^0\right)\nonumber\\
  &\triangleq {^{FC}_{0}}\mathbb{D}_t^{\alpha}u^{n},  \quad \text{for} \quad n>2,
  \label{defF}
\end{align}
where
\begin{align*}
&a_n=\alpha\Delta t^{\alpha}\sum_{j=1}^{\nexp}\omega_j e^{-ns_j\Delta t}\lambda_{j}^1,\quad &&b_n=\alpha\Delta t^{\alpha}\sum_{j=1}^{\nexp}\omega_j e^{-ns_j\Delta t}\lambda_{j}^2,\\
&\lambda_{j}^1=\frac{e^{-s_j \Delta t}}{s_j^2\Delta t}(e^{-s_j \Delta t}-1+s_j\Delta t), \quad &&\lambda_{j}^2=\frac{e^{-s_j \Delta t}}{s_j^2\Delta t}(1-e^{-s_j \Delta t}-e^{-s_j \Delta t}s_j\Delta t).
\end{align*}
Noting that $U_{\text{hist,i}}(t_1) = 0$ when $n=1$, we have
\begin{align} \label{defF1}
{^{FC}_{0}}\mathbb{D}_t^{\alpha}u^{1}&=\frac{\Delta t^{-\alpha}}{\Gamma(2-\alpha)}(u^1-u^{0}).
\end{align}
Recall that the L1-approximation (based on the linear interpolation of
the density function) of the Caputo derivative
${_{0}^{C}\!}D_t^{\alpha}u$ (see, for example, \cite{Lin:07,OS74})
is defined by the formula
\begin{align}
{^{C}_{0}}\!\mathbb{D}_t^{\alpha}u^{n}=\frac{ \dt^{-\alpha}}{\Gamma(2-\alpha)}\bigg[
a^{(\alpha)}_0u^{n}-\sum^{n-1}_{k=1}(a^{(\alpha)}_{n-k-1}-a^{(\alpha)}_{n-k})u^k-a^{(\alpha)}_{n-1}u^0\bigg],\label{Ddiscaputo}
\end{align}
where $a^{(\alpha)}_k:=(k+1)^{1-\alpha}-k^{1-\alpha}.$
The following lemma, which can be found in \cite{Sun:06}, established
an error bound for the L1-approximation \eqref{Ddiscaputo}.
\begin{lemma} [see \cite{Sun:06}]\label{lem4} Suppose that $f(t)\in C^2[0,t_n]$ and let
\begin{equation}\label{discaputo0}
R^{n}u:={_{0}^{C}\!}D_t^{\alpha}u(t)\big|_{t=t_{n}} - {^{C}_{0}} \!\mathbb{D}_t^{\alpha}u^{n},
\end{equation}
where $0<\alpha <1$. Then
\begin{equation}\label{Terror0}
 |R^n u|\leq\frac{\dt^{2-\alpha}}{\Gamma(2-\alpha)}\bigg(\frac{1-\alpha}{12}+\frac{2^{2-\alpha}}{2-\alpha}-(1+2^{-\alpha})\bigg)\max_{0\leq t \leq t_n}|u^{''}(t)|.
 \end{equation}
\end{lemma}

The following lemma provides an error bound for our approximation, denoted by
${^{FC}_{0}} \mathbb{D}_t^{\alpha}u^{n}$ in
\eqref{defF} and \eqref{defF1}.

\begin{lemma}\label{lemma6}Suppose that $u(t)\in C^2[0,t_n]$ and let
\begin{equation}\label{discaputo}
  {^{F}_{}}\!R^{n}u:={_{0}^{C}\!}D_t^{\alpha}u(t)\big|_{t=t_{n}} -
  {^{FC}_{0}} \mathbb{D}_t^{\alpha}u^{n},
\end{equation}
where $0<\alpha <1$. Then
\begin{equation}\label{Terror}
 |{^{F}_{}}\!R^n u|\leq\frac{\dt^{2-\alpha}}{\Gamma(2-\alpha)}\bigg(\frac{1-\alpha}{12}+\frac{2^{2-\alpha}}{2-\alpha}-(1+2^{-\alpha})\bigg)\max_{0\leq t \leq t_n}|u^{''}(t)|+\frac{\alpha\varepsilon t_{n-1}}{\Gamma(1-\alpha)}\max_{0\leq t \leq t_{n-1}}\left|u(t)\right|.
 \end{equation}
\end{lemma}
\begin{proof}
Obviously the only difference between our approximation
${^{FC}_{0}}\mathbb{D}_t^{\alpha}u^{n}$
and the L1-approximation ${^{C}_{0}}\mathbb{D}_t^{\alpha}u^{n}$ is that
the convolution kernel admits an absolute error bounded by $\varepsilon$
in its sum-of-exponentials approximation \eqref{2.2.1}, namely,
\begin{align}
  \left|{^{FC}_{0}}\mathbb{D}_t^{\alpha}u^{n}-{^{C}_{0}}\mathbb{D}_t^{\alpha}u^{n}\right|
  \leq\frac{\alpha\varepsilon}{\Gamma(1-\alpha)}\sum_{l=1}^{n-1}\int_{t_{l-1}}^{t_l}|\Pi_{1,l}u(s)|\mathrm{d}s.
\end{align}
where $\Pi_{1,l}u(t)=u(t_{l-1})\frac{t_l-t}{\Delta t}+u(t_{l})\frac{t-t_{l-1}}{\Delta t}$. And the triangle inequality leads to
\begin{equation}
|{^{F}_{}}\!R^n u|\leq|{^{C}_{}}\!R^n u|+\frac{\alpha\varepsilon}{\Gamma(1-\alpha)}
\sum_{l=1}^{n-1}\int_{t_{l-1}}^{t_l}|\Pi_{1,l}u(s)|\mathrm{d}s,
\end{equation}
where
\begin{align}\label{estimate1}
  \sum_{l=1}^{n-1}\int_{t_{l-1}}^{t_l}|\Pi_{1,l}u(s)|\mathrm{d}s
  \leq\max_{0\leq t \leq t_{n-1}}\left|u(t)\right|t_{n-1}.
\end{align}
Combining Lemma \ref{lem4} and \eqref{estimate1}, we obtain Lemma \ref{lemma6}.
\end{proof}

We also have the following useful inequality. The proof is given in
Appendix A.
\begin{lemma}\label{lemma3}
For any mesh functions $g=\{g^k|0\leq k \leq N\}$ defined on $\Omega_t$, the following inequality holds:
\begin{align*}
\Delta t \sum_{k=1}^n({^{FC}_{0}}\mathbb{D}_t^{\alpha}g^{k})g^k&\geq\frac{t_n^{-\alpha}-2\alpha\varepsilon t_{n-1}}{2\Gamma(1-\alpha)}\Delta t\sum_{k=1}^n(g^k)^2-\frac{t_n^{1-\alpha}-\alpha(1-\alpha)\varepsilon t_{n-1}\Delta t}{\Gamma(2-\alpha)}(g^0)^2 .
\end{align*}
\end{lemma}
%
%
\section{Application I: Linear Fractional Diffusion Equation}
Consider the following pure initial value problem of the linear fractional diffusion equation
\begin{align}
&_{0}^{C}\!D_{t}^{\alpha}u(x,t)= u_{xx}(x,t)+f(x,t),
&& x\in \mathbb{R}, \; t>0,\label{problemdifa}   \\
&u(x,0) =u_0(x),                  && x\in \mathbb{R}, \label{problemdifb} \\
&u(x,t) \rightarrow 0,                 &&    \text{when} \ |x|\rightarrow \infty,\label{problemdifc}
\end{align}
where the initial data $u_0$ and the source term $f(x,t)$ are assumed to be compactly supported in
the interval $\Omega_i :=  \{x| x_l<x<x_r\}$. To solve this problem
using a finite difference scheme, one needs to truncate the computational domain to a finite interval and impose some
boundary conditions at the end points, see \cite{Awotunde:15,Brunner:14,Brunner:15,Dea:13, Han:13, Gao12,Gao13,Ghaffari:13}. The exact nonreflecting boundary conditions for the above problem have
been derived in \cite{Gao12} via standard Laplace transform method and it is shown in \cite{Gao12} that
the above problem is equivalent to the following initial-boundary value problem
\begin{align}
&_{0}^{C}\!D_{t}^{\alpha}u(x,t)= u_{xx}(x,t)+f(x,t),   &x\in \Omega_i, \;t>0,   \label{exacta1}\\
&u(x,0)=u_0(x), & x\in\Omega_i, \label{exacta2}\\
&\frac{\partial u(x,t)}{\partial x}=\frac{1}{\Gamma(1-\frac{\alpha}{2})}\int_{0}^{t}\frac{u_s (x,s) }{(t-s)^{\frac{\alpha}{2}}}ds := {_{0}^{C}\!D}_{t}^{\frac{\alpha}{2}}u(x,t), & x=x_l,\\
&\frac{\partial u(x,t)}{\partial x}=-\frac{1}{\Gamma(1-\frac{\alpha}{2})}\int_{0}^{t}\frac{u_s (x,s) }{(t-s)^{\frac{\alpha}{2}}}d s := -{_{0}^{C}\!D}_{t}^{\frac{\alpha}{2}}u(x,t), &  x = x_r.\label{exacta3}
\end{align}
\subsection{Construction of the New Finite Difference Scheme}\label{NumericalScheme}
We now incorporate our fast evaluation scheme of the Caputo fractional derivative into the existing
finite difference scheme to construct a fast and stable FD scheme for solving the aforementioned IVP of the fractional
diffusion equation. We first introduce some standard notations.
For two given positive integers $N_T$ and $N_S$, let $\{t_n\}_{n=0}^{N_T}$
 be a equidistant partition of $[0,T]$ with
$t_n=n \dt$ and $ \dt=T/N_T$, and let $\{x_j\}_{j=0}^{N_S}$
 be a partition of $(x_l,x_r)$ with
$x_i=x_l+i h $ and $ h =(x_r-x_l)/N_S$. Denote
$u^n_i = u(x_i,t_n)$, $f_i^n = f(x_i,t_n)$, and
\begin{align*}
&\delta_{t}u^n_{i}=\frac{u_i^{n}-u^{n-1}_{i}}{ \dt},
&\delta_{x}u^n_{i+\frac{1}{2}}=\frac{u^n_{i+1}-u^n_{i}}{ h },\\
&\delta_x^{2}u^n_{i}=\frac{\delta_{x}u_{i+\frac{1}{2}}-\delta_{x}u_{i-\frac{1}{2}}}{ h },
&\tilde{\delta}_{x}u^n_{i}=\frac{u^n_{i+1}-u^n_{i-1}}{ 2h }.
\end{align*}

\begin{lemma}[see \cite{Sun:06}] \label{lem5} Suppose that $v\in C^3[x_l, x_r]$. Then
\begin{align}
&u_{xx}(x_0) -\frac{2}{h} \left[\delta_x u_{\frac{1}{2}}- u_x(x_0)\right] = -\frac{h}{3} u_{xxx}(x_0+\theta_1 h), \quad \theta_1 \in (0,1), \\
&u_{xx}(x_{N_S}) -\frac{2}{h} \left[u_x(x_{N_S})- \delta_x u_{{N_S}-
\frac{1}{2}} \right] = \frac{h}{3} u_{xxx}(x_{N_S} -\theta_2 h), \quad \theta_2 \in (0,1).
\end{align}
\end{lemma}
The finite difference scheme in \cite{Gao12} for the problem \eqref{exacta1}-\eqref{exacta3} can be written
in the following form
\begin{align}
  &{^{C}_{0}}\!\mathbb{D}
  _t^{\alpha}u^{n}_{i}= \delta_x^2u_{i}^{n}+f^{n}_{i},  \quad 1\leq i\leq N_S-1~,1\leq n\leq N_T,    \label{schm1}\\
&{^{C}_{0}}\!\mathbb{D}_t^{\alpha}u^{n}_{0}=\frac{2}{ h }\left[ \delta_x u^{n}_{\frac{1}{2}}-{^{C}_{0}}\!\mathbb{D}_t^{\alpha/2}u^{n}_{0}\right]+f^n_{0},\label{schm2}\\
&{^{C}_{0}}\!\mathbb{D}_t^{\alpha}u^{n}_{N_S}=\frac{2}{ h }\left[-\delta_x u^{n}_{{N_S}-\frac{1}{2}} -{^{C}_{0}}\!\mathbb{D}_t^{\alpha/2}u^{n}_{N_S}\right]+f^n_{{N_S}},\label{schm3}\\
& u_i^{0}=u_0(x_i),  \quad  0\leq i\leq {N_S}.\label{schm4}
\end{align}
Replacing the standard L1-approximation ${^{C}_{0}}\!\mathbb{D}$
for the Caputo derivative by our fast evaluation scheme ${^{FC}_{0}}\!\mathbb{D}$, we obtain a
fast FD scheme of the following form
\begin{align}
&{^{FC}_{0}}\!\mathbb{D}_t^{\alpha}u^{n}_{i}= \delta_x^2u_{i}^{n}+f^{n}_{i},  \quad 1\leq i\leq {N_S}-1~,1\leq n\leq N_T,    \label{Fschm1}\\
&{^{FC}_{0}}\!\mathbb{D}_t^{\alpha}u^{n}_{0}=\frac{2}{ h }\left[ \delta_xu^{n}_{\frac{1}{2}}-{^{FC}_{0}}\mathbb{D}_t^{\alpha/2}u^{n}_{0}\right]+f^n_{0},\label{Fschm2}\\
&{^{FC}_{0}}\!\mathbb{D}_t^{\alpha}u^{n}_{{N_S}}=\frac{2}{ h }\left[-\delta_x u^{n}_{{N_S}-\frac{1}{2}} -{^{FC}_{0}}\mathbb{D}_t^{\alpha/2}u^{n}_{{N_S}}\right]+f^n_{{N_S}},\label{Fschm3}\\
& u_i^{0}=u_0(x_i),  \quad  0\leq i\leq {N_S}.\label{Fschm4}
\end{align}
\subsection{Stability and Error Analysis of the New Scheme}
Let $\mathcal{S}_h = \{u|u=(u_0,u_1,\cdots,u_{N_s})\}$. We first recall an elementary property of
the mesh function $u\in \mathcal{S}_h$.
\begin{lemma}[\cite{Gao12}]\label{8}
For any mesh function $u$ defined on $ \mathcal{S}_h$, the following inequality holds
\begin{align*}
\left\|u\right\|^2_{\infty}\leq\theta \left\|\delta_x u\right\|^2+(\frac{1}{\theta}+\frac{1}{L})\left\|u\right\|^2,\quad \forall \theta >0,
\end{align*}
where $L$ is the length of the computational domain and here $L = x_r - x_l$.
\end{lemma}

We now show the following prior estimate holds for the solution of the new FD scheme.
\begin{theorem}[Prior Estimate]\label{theorem1}
Suppose $\{u_i^k|0\leq i\leq {N_S}, 0\leq k\leq N_T\}$ is the solution of the finite difference scheme \eqref{Fschm1}--\eqref{Fschm4}.
Then for any $1\leq n\leq N_T$,
\begin{align} \label{priest}
\Delta t\sum_{k=1}^n\left\|u^k\right\|_{\infty}^2 \leq & \frac{2\big(1+\sqrt{1+L^2\mu}\big)}{L\mu}\bigg(\rho\left\|u^0\right\|^2+\varrho\left[(u_0^0)^2+(u_{N_S}^0)^2\right]\nonumber\\
&+\frac{\Delta t}{8\nu}\sum_{k=1}^n\left[(hf_0^k)^2+(hf_{N_S}^k)^2\right]+\frac{\Delta t}{\mu}\sum_{k=1}^nh\sum_{i=1}^{{N_S}-1}(f_i^k)^2\bigg),
\end{align}
where
\begin{equation}\label{rhomunu}
\begin{aligned}
  &\rho=\frac{t_n^{1-\alpha}-\alpha(1-\alpha)\varepsilon t_{n-1}\Delta t}{\Gamma(2-\alpha)},
  &\mu=\frac{t_n^{-\alpha}-2\alpha\varepsilon t_{n-1}}{\Gamma(1-\alpha)},\\
  &\varrho=\frac{t_n^{1-\frac{\alpha}{2}}-\frac{\alpha}{2}(1-\frac{\alpha}{2})\varepsilon t_{n-1}\Delta t}{\Gamma(2-\frac{\alpha}{2})},
  &\nu=\frac{t_n^{-\frac{\alpha}{2}}-\alpha\varepsilon t_{n-1}}{\Gamma(1-\frac{\alpha}{2})}.
\end{aligned}
\end{equation}

\end{theorem}

\begin{proof}
Multiplying $hu_i^k$ on both sides of (\ref{Fschm1}), and summing up for $i$ from 1 to ${N_S}-1$, we have
\[h\sum_{i=1}^{{N_S}-1}({^{FC}_{0}}\!\mathbb{D}_t^{\alpha}u^{k}_{i})u^{k}_{i}-h\sum_{i=1}^{{N_S}-1}(\delta_x^2u_i^k)u_i^k=h\sum_{i=1}^{{N_S}-1}f_i^ku_i^k.\]
Multiplying $\frac{h}{2}u_0^k$ and $\frac{h}{2}u_{N_S}^k$ on both sides of (\ref{Fschm2}) and (\ref{Fschm3}), respectively, then adding the results with the above identity, we obtain
\begin{align}\label{5}
&({^{FC}_{0}}\!\mathbb{D}_t^{\alpha}u^{k},u^{k})+\left[-(\delta_xv_{\frac{1}{2}}^k)u_0^k-h\sum_{i=1}^{{N_S}-1}(\delta_x^2u_i^k)u_i^k+(\delta_xu_{{N_S}-\frac{1}{2}}^k)u_{{N_S}}^k\right]\\ \nonumber
&+({^{FC}_{0}}\!\mathbb{D}_t^{\frac{\alpha}{2}}u_0^{k},u_0^{k})+({^{FC}_{0}}\!\mathbb{D}_t^{\frac{\alpha}{2}}u_{N_S}^{k},u_{N_S}^{k}) =\frac{1}{2}(hf_0^k)u_0^k+h\sum_{i=1}^{{N_S}-1}f_i^ku_i^k+\frac{1}{2}(f_{N_S}^k)u_{N_S}^k.
\end{align}
Observing the summation by parts, we have
\begin{align}\label{4}
-(\delta_xv_{\frac{1}{2}}^k)u_0^k-h\sum_{i=1}^{{N_S}-1}(\delta_x^2u_i^k)u_i^k+(\delta_xu_{{N_S}-\frac{1}{2}}^k)u_{{N_S}}^k=\left\|\delta_xu^k\right\|^2.
\end{align}
Substituting (\ref{4}) into (\ref{5}), and multiplying $\Delta t$ on both sides of the resulting identity, and summing up for $k$ from 1 to $n$, it follows from Lemma \ref{lemma3} that
\begin{align}
&\Delta t\frac{t_n^{-\alpha}-2\alpha\varepsilon t_{n-1}}{2\Gamma(1-\alpha)}\sum_{k=1}^n\left\|u^k\right\|^2+\Delta t\frac{t_n^{-\frac{\alpha}{2}}-\alpha\varepsilon t_{n-1}}{2\Gamma(1-\frac{\alpha}{2})}\sum_{k=1}^n\left[(u_0^k)^2+(u_{N_S}^k)^2\right]+\Delta t\sum_{k=1}^n\left\|\delta_xu^k\right\|^2\nonumber\\
\leq&\frac{t_n^{1-\alpha}-\alpha(1-\alpha)\varepsilon t_{n-1}\Delta t}{\Gamma(2-\alpha)}\left\|u^0\right\|^2+\frac{t_n^{1-\frac{\alpha}{2}}-\frac{\alpha}{2}(1-\frac{\alpha}{2})\varepsilon t_{n-1}\Delta t}{\Gamma(2-\frac{\alpha}{2})}\left[(u_0^0)^2+(u_{N_S}^0)^2\right]\nonumber\\
&+\Delta t\sum_{k=1}^n\left[\frac{1}{2}(hf_0^k)u_0^k+h\sum_{i=1}^{{N_S}-1}f_i^ku_i^k+\frac{1}{2}(f_{N_S}^k)u_{N_S}^k\right].\label{7}
\end{align}
Applying the Cauchy-Schwarz inequality, we obtain
\begin{align}
&\frac{1}{2}(hf_0^k)u_0^k+h\sum_{i=1}^{{N_S}-1}f_i^ku_i^k+\frac{1}{2}(f_{N_S}^k)u_{N_S}^k\nonumber\\
\leq&\frac{t_n^{-\frac{\alpha}{2}}-\alpha\varepsilon t_{n-1}}{2\Gamma(1-\frac{\alpha}{2})}\left[(u_0^k)^2+(u_{N_S}^k)^2\right]+\frac{\Gamma(1-\frac{\alpha}{2})}{8(t_n^{-\frac{\alpha}{2}}-\alpha\varepsilon t_{n-1})}\left[(hf_0^k)^2+(hf_{N_S}^k)^2\right]\nonumber\\
&+h\sum_{i=1}^{{N_S}-1}\left[\frac{t_n^{-\alpha}-2\alpha\varepsilon t_{n-1}}{4\Gamma(1-\alpha)}(u_i^k)^2+\frac{\Gamma(1-\alpha)}{t_n^{-\alpha}-2\alpha\varepsilon t_{n-1}}(f_i^k)^2\right]\nonumber\\
\leq&\frac{t_n^{-\frac{\alpha}{2}}-\alpha\varepsilon t_{n-1}}{2\Gamma(1-\frac{\alpha}{2})}\left[(u_0^k)^2+(u_{N_S}^k)^2\right]+\frac{\Gamma(1-\frac{\alpha}{2})}{8(t_n^{-\frac{\alpha}{2}}-\alpha\varepsilon t_{n-1})}\left[(hf_0^k)^2+(hf_{N_S}^k)^2\right]\nonumber\\
&+\frac{t_n^{-\alpha}-2\alpha\varepsilon t_{n-1}}{4\Gamma(1-\alpha)}\left\|u^k\right\|^2+h\sum_{i=1}^{{N_S}-1}\frac{\Gamma(1-\alpha)}{t_n^{-\alpha}-2\alpha\varepsilon t_{n-1}}(f_i^k)^2. \label{6}
\end{align}
The substitution of (\ref{6}) into (\ref{7}) produces
\begin{align}\label{10}
\frac{\mu}{4}\Delta t\sum_{k=1}^n\left\|u^k\right\|^2&+\Delta t\sum_{k=1}^n\left\|\delta_xu^k\right\|^2
\leq\rho\left\|u^0\right\|^2+\varrho\left[(u_0^0)^2+(u_{N_S}^0)^2\right] \nonumber\\ 
&+\frac{\Delta t}{8\nu}\sum_{k=1}^n\left[(hf_0^k)^2+(hf_{N_S}^k)^2\right]+\frac{\Delta t}{\mu}\sum_{k=1}^nh\sum_{i=1}^{{N_S}-1}(f_i^k)^2.
\end{align}
Taking $\theta>0$ such that $\frac{1\slash\theta+1\slash L}{\theta}=\frac{\mu}{4}$ (i.e., $\theta=2\Big(1+\sqrt{1+L^2\mu}\Big)\slash(L\mu)$), and following from Lemma \ref{8}, we have
\begin{align}\label{9}
\Delta t\sum_{k=1}^n\left\|u^k\right\|^2_{\infty}\leq\frac{2\big(1+\sqrt{1+L^2\mu}\big)}{L\mu}\bigg(\frac{\mu}{4}\Delta t\sum_{k=1}^n\left\|u^k\right\|^2+\Delta t\sum_{k=1}^n\left\|\delta_xu^k\right\|^2\bigg).
\end{align}
Combining (\ref{10}) with (\ref{9}), we obtain the inequality \eqref{priest}.
\end{proof}

The priori estimate leads to the stability of the new FD scheme.
\begin{theorem}[Stability]
The scheme (\ref{Fschm1})-(\ref{Fschm4}) is unconditionally stable for any given compactly supported initial data and source term.
\end{theorem}

We now present an error analysis of the new scheme \eqref{Fschm1}-\eqref{Fschm4}.
\begin{theorem}[Error Analysis] Suppose $u(x,t)\in C_{x,t}^{4,2}(\left[x_l,x_r\right]\times[0,T])$ and $\{u_i^k|0\leq i\leq {N_S}, 0\leq k \leq N_T\}$ are solutions of the problem (\ref{exacta1})-(\ref{exacta3}) and the difference scheme (\ref{Fschm1})-(\ref{Fschm4}), respectively. Let $e^k_i = u_i^k - u(x_i,t_k)$. Then there exists a positive constant $c_2$ such that
\begin{align}\label{convergence}
\sqrt{\Delta t\sum_{k=1}^n\|e^k\|_{\infty}^2 }\leq c_2(h^2+\Delta t^{2-\alpha} +\varepsilon),\quad 1\leq n\leq N_T,
\end{align}
where
$c_2^2=\frac{4c_1^2T\big(1+\sqrt{1+L^2\mu}\big)}{L\mu}\Big(\frac{1}{\nu}+\frac{L}{\mu}\Big)$
with $c_1$ is a positive constant {\rm{(}}see \eqref{11}-\eqref{13}{\rm{)}}, and $\mu$, $\nu$ are defined in \eqref{rhomunu}.
\end{theorem}
\begin{proof}
We observe that the error $e^k_i$ satisfies the following FD scheme:
\begin{align}
&{^{FC}_{0}}\!\mathbb{D}_t^{\alpha}e^{k}_{i}= \delta_x^2e_{i}^{k}+T^{k}_{i},  \quad 1\leq i\leq {N_S}-1~,1\leq k\leq N_T,    \label{eq1}\\
&{^{FC}_{0}}\!\mathbb{D}_t^{\alpha}e^{k}_{0}=\frac{2}{ h }\left[ \delta_xe^{k}_{\frac{1}{2}}-{^{FC}_{0}}\mathbb{D}_t^{\alpha/2}e^{k}_{0}\right]+T^k_{0},\label{eq2}\\
&{^{FC}_{0}}\!\mathbb{D}_t^{\alpha}e^{k}_{{N_S}}=\frac{2}{ h }\left[-\delta_x e^{n}_{{N_S}-\frac{1}{2}} -{^{FC}_{0}}\mathbb{D}_t^{\alpha/2}e^{k}_{{N_S}}\right]+T^k_{{N_S}},\label{eq3}\\
& e_i^{0}=0,  \quad  0\leq i\leq {N_S}.\label{eq4}
\end{align}
where the truncation terms $T^k$ at the interior and boundary points are given by the formulas
\begin{align*}
&T_i^k=-\left[{_{0}^{C}\!}D_t^{\alpha}u(x_i,t_k)-{^{FC}_{0}} \mathbb{D}_t^{\alpha}U_i^k\right]+\left[u_{xx}(x_i,t_k)-\delta_x^2U_i^k\right],~~1\leq i\leq {N_S},~1\leq k\leq N_T,\\
&T_0^k=\left\{u_{xx}(x_0,t_k)-\frac{2}{h}\left[\delta_xU_{\frac{1}{2}}^k-u_x(x_0,t_k)\right]-\frac{2}{h}\left[{_{0}^{C}\!}D_t^{\frac{\alpha}{2}}u(x_0,t_k)-{^{FC}_{0}} \mathbb{D}_t^{\frac{\alpha}{2}}U_0^k\right]\right\}\nonumber\\
&\qquad-\left[{_{0}^{C}\!}D_t^{\alpha}u(x_0,t_k)-{^{FC}_{0}} \mathbb{D}_t^{\alpha}U_0^k\right],~1\leq k\leq N_T,\\
&T_{N_S}^k=\left\{u_{xx}(x_{N_S},t_k)-\frac{2}{h}\left[u_x(x_{N_S},t_k)-\delta_xU_{{N_S}-\frac{1}{2}}^k\right]-\frac{2}{h}\left[{_{0}^{C}\!}D_t^{\frac{\alpha}{2}}u(x_{N_S},t_k)-{^{FC}_{0}} \mathbb{D}_t^{\frac{\alpha}{2}}U_{N_S}^k\right]\right\}\nonumber\\
&\qquad-\left[{_{0}^{C}\!}D_t^{\alpha}u(x_{N_S},t_k)-{^{FC}_{0}} \mathbb{D}_t^{\alpha}U_{N_S}^k\right],~1\leq k\leq N_T.
\end{align*}
Using Lemma \ref{lem5} and Taylor expansion, it is easy to show that the truncation terms $T^k$
satisfy the following error bounds
\begin{align}
&\left|T_i^k\right|\leq c_1(\Delta t^{2-\alpha}+h^2+{\varepsilon}),\quad 1\leq i\leq {N_S}-1,\quad 1\leq k\leq N_T,\label{11}\\
&\left|T_0^k\right|\leq c_1(\Delta t^{2-\alpha}+h+\frac{\Delta t^{2-\alpha\slash 2}}{h}+\frac{\varepsilon}{h}), \quad 1\leq k\leq N_T,\label{12}\\
&\left|T_{N_S}^k\right|\leq c_1(\Delta t^{2-\alpha}+h+\frac{\Delta t^{2-\alpha\slash 2}}{h}+\frac{\varepsilon}{h}), \quad 1\leq k\leq N_T\label{13}
\end{align}
with $c_1$ some positive constant.
Thus, for $h\leq 1$ and $\Delta t\leq 1$, we have 
\begin{align}\label{convergenceeq2}
&\frac{1}{4\nu}\left[(hT_0^k)^2 +(hT_{N_S}^k)^2\right]+\frac{2}{\mu}h\sum_{i=1}^{{N_S}-1}(T_i^k)^2\nonumber\\
\leq&\frac{c_1^2}{2\nu}\Big(h\Delta t^{2-\alpha}+\Delta t^{2-\frac{\alpha}{2}}+h^2+{\varepsilon}\Big)^2+\frac{2c_1^2L}{\mu}\Big(\Delta t^{2-\alpha}+h^2+{\varepsilon}\Big)^2\nonumber\\
\leq&\frac{2c_1^2}{\nu}\Big(\Delta t^{2-\alpha}+h^2+{\varepsilon}\Big)^2+\frac{2c_1^2L}{\mu}\Big(\Delta t^{2-\alpha}+h^2+{\varepsilon}\Big)^2\nonumber\\
\leq&4c_1^2(\frac{1}{\nu}+\frac{L}{\mu})(\Delta t^{2-\alpha}+h^2)+4c_1^2(\frac{1}{\nu}+\frac{L}{\mu}){\varepsilon}^2
\end{align}
A direct application of Theorem \ref{theorem1} to the system (\ref{eq1})-(\ref{eq4}) produces
\begin{align}\label{convergenceeq1}
\Delta t\sum_{k=1}^n\|e^k\|_{\infty}^2 \leq\frac{\Delta t(1+\sqrt{1+L^2\mu})}{L\mu}\sum_{k=1}^n(\frac{1}{4\nu}\left[(hT_0^k)^2+(hT_{N_S}^k)^2\right]+\frac{2}{\mu}h\sum_{i=1}^{{N_S}-1}(T_i^k)^2).
\end{align}
Substituting (\ref{convergenceeq2}) into (\ref{convergenceeq1}), simplifying the resulting expressions, and
taking the square root for both sides, we obtain \eqref{convergence}.
\end{proof}
\begin{table}[!ht]
  \caption{The errors and convergence orders in time with fixed spatial size $h=\pi\slash20000$ for the fast scheme (\ref{Fschm1})-(\ref{Fschm4}) and the direct scheme (\ref{schm1})-(\ref{schm4}). } \label{tab1}
   \centering
    \begin{adjustwidth}{0cm}{0cm}
  \resizebox{!}{1.4cm}
  {
\begin{tabular}{cccccccccccc}
   \hline
 \multicolumn{1}{c}{\multirow{3}{0.5cm}{$\Delta t$}}& \multicolumn{5}{c}{$\alpha=0.2$} &\multicolumn{1}{c}{\multirow{1}{0cm}{}}  &\multicolumn{5}{c}{$\alpha=0.5$}\\
 \cline{2-6}\cline{8-12}\multicolumn{1}{c}{} &\multicolumn{2}{c}
       {Fast scheme} &\multicolumn{1}{c}{\multirow{2}{0.1cm}{}}  &\multicolumn{2}{c}{Direct scheme}&\multicolumn{1}{c}{\multirow{2}{0cm}{}}&\multicolumn{2}{c}{Fast scheme} &\multicolumn{1}{c}{\multirow{2}{0cm}{}}  &\multicolumn{2}{c}{Direct scheme}\\
\cline{2-3}  \cline{5-6}\cline{8-9} \cline{11-12}\multicolumn{1}{c}{}&\multicolumn{1}{c} {$E(h,\tau)$} & \multicolumn{1}{c}{$r_t$}&\multicolumn{1}{c}{}  &\multicolumn{1}{c} {$E(h,\tau)$} & \multicolumn{1}{c}{$r_t$} &\multicolumn{1}{c}{} &\multicolumn{1}{c} {$E(h,\tau)$} & \multicolumn{1}{c}{$r_t$}&\multicolumn{1}{c}{}  &\multicolumn{1}{c} {$E(h,\tau)$} & \multicolumn{1}{c}{$r_t$}\\ \hline
 $1\slash10$&1.570e-02&1.70&&1.570e-02&1.70&&8.151e-02&1.48&&8.151e-02&1.48\\
 $1\slash20$&4.846e-03&1.70&&4.844e-03&1.70&&2.922e-02&1.47&&2.923e-02&1.47\\
 $1\slash40$&1.489e-03&1.72&&1.488e-03&1.71&&1.052e-02&1.48&&1.052e-02&1.48\\
 $1\slash80$&4.524e-04&1.70&&4.541e-04&1.72&&3.781e-03&1.48&&3.784e-03&1.48\\
 $1\slash160$&1.395e-04&   &&1.375e-04&    &&1.357e-03&    &&1.357e-03&    \\\hline
    \end{tabular}
    }
    \end{adjustwidth}
\end{table}
\begin{table}[!ht]
  \caption{The errors, convergence orders in space, and CPU time in seconds
    with fixed temporal step size $\Delta t=1\slash30000$ for the fast scheme (\ref{Fschm1})-(\ref{Fschm4}) and the direct scheme (\ref{schm1})-(\ref{schm4}). }
    \label{tab2}
   \centering
    \begin{adjustwidth}{0cm}{0cm}
  \resizebox{!}{1.55cm}
  {
\begin{tabular}{cccccccccccc}
   \hline
 \multicolumn{1}{c}{\multirow{3}{0.5cm}{$h$}}& \multicolumn{5}{c}{$\alpha=0.2$} &\multicolumn{1}{c}{\multirow{1}{0cm}{}}  &\multicolumn{5}{c}{$\alpha=0.5$}\\
 \cline{2-6}\cline{8-12}\multicolumn{1}{c}{} &\multicolumn{2}{c}
       {Fast scheme } &\multicolumn{1}{c}{\multirow{2}{0.1cm}{}}  &\multicolumn{2}{c}{Direct scheme}&\multicolumn{1}{c}{\multirow{2}{0cm}{}}&\multicolumn{2}{c}{Fast scheme} &\multicolumn{1}{c}{\multirow{2}{0cm}{}}  &\multicolumn{2}{c}{Direct scheme}\\
\cline{2-3}  \cline{5-6}\cline{8-9} \cline{11-12}\multicolumn{1}{c}{}&\multicolumn{1}{c} {$E(h,\tau)$} & \multicolumn{1}{c}{$r_s$}&\multicolumn{1}{c}{}  &\multicolumn{1}{c} {$E(h,\tau)$} & \multicolumn{1}{c}{$r_s$} &\multicolumn{1}{c}{} &\multicolumn{1}{c} {$E(h,\tau)$} & \multicolumn{1}{c}{$r_s$}&\multicolumn{1}{c}{}  &\multicolumn{1}{c} {$E(h,\tau)$} & \multicolumn{1}{c}{$r_s$}\\ \hline
  $\pi\slash10$&8.862e-01&2.06&&8.632e-01&2.03&&8.652e-01&2.05&&8.035e-01&1.99\\
   $\pi\slash20$&2.122e-01&2.01&&2.107e-01&2.01&&2.087e-01&2.01&&2.016e-01&1.99\\
$\pi\slash40$&5.258e-02&2.00&&5.244e-02&2.00&&5.177e-02&2.00&&5.064e-02&1.99\\
  $\pi\slash80$&1.312e-02&2.00&&1.311e-02&2.00&&1.292e-02&2.00&&1.272e-02&1.99\\
   $\pi\slash160$&3.280e-03&&&3.277e-03&&&3.229e-03&&&3.195e-03&\\\hline
  CPU(s)&\multicolumn{2}{c}{43.37}&&\multicolumn{2}{c}{3304.66}&&\multicolumn{2}{c}{43.65}&&\multicolumn{2}{c}{2226.06}\\ \hline
    \end{tabular}
    }
    \end{adjustwidth}
\end{table}

\begin{figure}[htbp]
\centering
\includegraphics[width=0.45\textwidth]{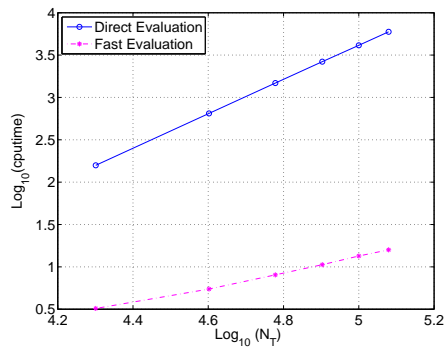}
\caption{The log-log (in base $10$) plot of the CPU time (in seconds) versus the
total number of time steps $N_T$ for Application I. Here $N_S=30$ and $\alpha=0.5$.
} \label{fig1}
\end{figure}

\subsection{Numerical Results}\label{NumericalExperiments}
To test the convergence rates of the new scheme, we take
the computational domain $\Omega_i = [0,\pi]$, and set
\begin{align*}
&f(x,t)=
\begin{cases}
\Gamma(4+\alpha)x^4(\pi-x)^4\exp(-x)t^3\slash 6-x^2(\pi-x)^2\{t^{3+\alpha}\exp(-x)\\
\left[x^2(56-16x+x^2)-2\pi x(28-12x+x^2)+\pi^2(12-8x+x^2)\right]\\
+4(3\pi^2-14\pi x+14x^2)\},\quad(x,t)\in\Omega_i,\\
0,\quad(x,t)\notin\Omega_i,
\end{cases}
& \notag\\
&u_0(x)=
\begin{cases}
x^4(\pi-x)^4,\quad(x,t)\in\Omega_i,\\
0,\quad(x,t)\notin\Omega_i.
\end{cases}
\end{align*}
It is known that the IVP \eqref{problemdifa}--\eqref{problemdifc} has the exact solution given by the formula
\begin{align}\label{exactsol}
u(x,t)=
x^4(\pi-x)^4\left[\exp(-x)t^{3+\alpha}+1\right],\quad(x,t)\in\Omega_i \times (0,T].
\end{align}
To illustrate the performance of the numerical scheme, we define the maximum norm of the error and the convergence rates
with respect to temporal and spatial sizes, respectively by the formulas
\[E(h,\dt)=\sqrt{\dt\sum_{k=1}^N\left\|e^k\right\|_{\infty}^2},\quad r_t=\log_2\frac{E(h,\dt)}{E(h,\dt\slash2)},\quad
r_s=\log_2\frac{E(h,\dt)}{E(h\slash2,\dt)},\]
where the error $e^k$ is measured against the exact solution \eqref{exactsol}.

First, we check the convergence rate of the new scheme in time. We fix the spatial mesh size $h=\frac{\pi}{20000}$ and
refine the temporal mesh size $\dt$ from $\frac{1}{10}$ to $\frac{1}{160}$. Obviously, $h$ is chosen so small that the error due
to spatial discretization is negligible. The precision for the sum-of-exponentials approximation
for the convolution kernel is set to $\varepsilon=10^{-7}$. Table \ref{tab1} shows the numerical results for
two different fractional values: $\alpha=0.2 $ and $\alpha=0.5$. Next, we fix the temporal mesh size
$\dt=\frac{1}{30000}$ so that the error due to temporal discretization is neglible. We then change the spatial
step size $h$ from $\frac{\pi}{10}$ to $\frac{\pi}{160}$ to check the convergence order of the new scheme in space. Table \ref{tab2}
shows the numerical results for $\alpha=0.2 $ and $\alpha=0.5$. Table \ref{tab1} shows that the convergence order
in time is ${O}(\dt^{2-\alpha})$ for both the direct scheme \eqref{schm1}-\eqref{schm2} and our fast scheme
\eqref{Fschm1}-\eqref{Fschm2}. While Table \ref{tab2} shows that the convergence order
in space is ${O}(h^{2})$ for both the direct scheme and our fast scheme.

To demonstrate the complexity of the two schemes, we plot in Fig \ref{fig1} the
CPU time of the two schemes in seconds. We observe that while the direct scheme
scales like $O(N_T^2)$, the CPU time increases
almost linearly with the total number of time steps $N_T$ for the fast scheme.
There is a significant speed-up in fast scheme as compared with the
direct scheme even for $N_T$ of modest size.
\section{Application II: Nonlinear Fractional Diffusion Equation}
Consider now the initial value problem of the nonlinear fractional diffusion equation of the form
\begin{equation}\label{5.1}
\begin{aligned}
  _{0}^{C}\!D_{t}^{\alpha} u(x,t)&=u_{xx} + f(u),\\
  u(x,0)&= u_0(x).
\end{aligned}
\end{equation}
This problem has rich applications. When $f(u)=-u(1-u)$, \eqref{5.1} is the Fisher equation,
which is used to model the spatial and
temporal propagation of a virile gene in an infinite
medium \cite{fis37}, the chemical kinetics \cite{mal92},
flame propagation \cite{fra55}, and many other scientific problems \cite{mer12}.
When $f(u)=-0.1 u(1-u)(u-0.001)$, \eqref{5.1} is the time-fractional Huxley equation, which is used to
describe the transmission of nerve impulses \cite{fit61, nag62} with many applications in biology and
the population genetics in circuit theory \cite{shi05}.

When the initial data
$u_0(x)$ is compactly supported on $\Omega_i=[x_l,x_r]$,
the following finite difference scheme with artificial boundary conditions imposed
on two end points has been proposed
in \cite{LZ15} to solve the problem \eqref{5.1}
\begin{align}
\label{3.8}
&{^{C}_{0}}\mathbb{D}_t^{\alpha} U_i^{n}= \delta_x^2 U_i^{n} + f(U_i^{n-1}),\qquad\qquad \qquad 1\leq i \leq M-1, \\
&(\tilde{\delta}_x + 3s_{0}^{\frac{\alpha}{2}}) {^{C}_{0}}\mathbb{D}_t^\alpha U_{M-1}^{n} + (3s^\alpha_0\tilde{\delta}_x+ s_{0}^{\frac{3\alpha}{2}} ) {U_{M-1}^{n}} = (\tilde{\delta}_x + 3s_{0}^{\frac{\alpha}{2}}) f(U_{M-1}^{n-1}),\label{3.10}\\
&(\tilde{\delta}_x - 3s_{0}^{\frac{\alpha}{2}}){^{C}_{0}}\mathbb{D}_t^\alpha U_{1}^{n-1} + (3s^\alpha_0 \tilde{\delta}_x- s_{0}^{\frac{3\alpha}{2}} ) {U_{1}^{n-1}} = (\tilde{\delta}_x - 3s_{0}^{\frac{\alpha}{2}}) f(U_{1}^{n-1}).  \label{3.11}
\end{align}
Under the assumption that $f(u)\in C^2([0,T])$, it has been shown in \cite{LZ15} that the scheme \eqref{3.8}
the convergence rate of ${O}(h^2+\Delta t)$ in $L_{\infty}$ norm, defined by
$\left\|e\right\|_{\infty}=\max\limits_{0\leq i \leq M}\left|e_i\right|$. With the L1-approximation ${^{C}_{0}}\mathbb{D}_t^\alpha$
replaced by our fast evaluation scheme ${^{FC}_{0}}\mathbb{D}_t^\alpha$, we obtain a fast scheme for solving \eqref{5.1}, which is as follows:
\begin{align}
\label{3.9f}
&{^{FC}_{0}}\mathbb{D}_t^{\alpha} U_i^{n}= \delta_x^2 U_i^{n} + f(U_i^{n-1}),\qquad\qquad \qquad 1\leq i \leq M-1, \\
&(\tilde{\delta}_x + 3s_{0}^{\frac{\alpha}{2}}) {^{FC}_{0}}\mathbb{D}_t^\alpha U_{M-1}^{n} + (3s^\alpha_0\tilde{\delta}_x+ s_{0}^{\frac{3\alpha}{2}} ) {U_{M-1}^{n}} = (\tilde{\delta}_x + 3s_{0}^{\frac{\alpha}{2}}) f(U_{M-1}^{n-1}),\label{3.10f}\\
&(\tilde{\delta}_x - 3s_{0}^{\frac{\alpha}{2}}){^{FC}_{0}}\mathbb{D}_t^\alpha U_{1}^{n-1} + (3s^\alpha_0 \tilde{\delta}_x- s_{0}^{\frac{3\alpha}{2}} ) {U_{1}^{n-1}} = (\tilde{\delta}_x - 3s_{0}^{\frac{\alpha}{2}}) f(U_{1}^{n-1}).  \label{3.11f}
\end{align}

\subsection{Numerical Examples}
We will give two examples - the Fisher equation and the Huxley equation to illustrate the performance of our scheme. For both examples,
in order to investigate the convergence orders of our scheme, the reference solution is computed over a large interval
$\Omega_r = [-12,12]$ with very small mesh sizes $h=2^{-10}$, and $\Delta t=2^{-14}$. We then set
$\Omega_i = [-6,6]$ and the precision for the sum-of-exponentials approximation of the convolution kernel to
$\varepsilon=10^{-9}$. The temporal step size is fixed at $\Delta t=2^{-14}$ when testing the order of convergence
in space; and the spatial step size is fixed at $h = 2^{-10}$ when testing the order of convergence
in time.
\begin{example}\label{exm2}
\end{example} We consider the time fractional Fisher equation
\begin{equation}\label{fisher}
  _0^C\!{D}_t^\alpha u(x,t)= u_{xx} -u(1-u),~~0< t\leq 1
\end{equation}
with the double Gaussian initial value
$u(x,0)=\exp(-10(x-0.5)^2)+\exp(-10(x+0.5)^2)$. Tables \ref{tab3} and \ref{tab4}
present the numerical results for $\alpha = 0.2,\ 0.5$, which
show that our fast scheme (\ref{3.9f})-(\ref{3.11f}) has
the same convergence order ${O}(h^2+\Delta t)$ in $L_{\infty}$ norm as the direct scheme (\ref{3.8})-(\ref{3.11}),
but takes much less computational time.
\begin{table}[!ht]
  \caption{The errors and convergence orders in time for the fractional Fisher equation \eqref{fisher} with fixed spatial step size $h=2^{-10}$.} \label{tab3}
   \centering
    \begin{adjustwidth}{0cm}{0cm}
  \resizebox{!}{1.36cm}
  {
\begin{tabular}{cccccccccccc}
   \hline
 \multicolumn{1}{c}{\multirow{3}{0.5cm}{$\Delta t$}}& \multicolumn{5}{c}{$\alpha=0.2$} &\multicolumn{1}{c}{\multirow{1}{0cm}{}}  &\multicolumn{5}{c}{$\alpha=0.5$}\\
 \cline{2-6}\cline{8-12}\multicolumn{1}{c}{} &\multicolumn{2}{c}{Fast scheme
 } &\multicolumn{1}{c}{\multirow{2}{0.1cm}{}}  &\multicolumn{2}{c}{Direct scheme
   }&\multicolumn{1}{c}{\multirow{2}{0cm}{}}&\multicolumn{2}{c}{Fast scheme} &\multicolumn{1}{c}{\multirow{2}{0cm}{}}  &\multicolumn{2}{c}{Direct scheme}\\
\cline{2-3}  \cline{5-6}\cline{8-9} \cline{11-12}\multicolumn{1}{c}{}&\multicolumn{1}{c} {$\left\|e^n\right\|_{\infty}$} & \multicolumn{1}{c}{$r_t$}&\multicolumn{1}{c}{}  &\multicolumn{1}{c} {$\left\|e^n\right\|_{\infty}$} & \multicolumn{1}{c}{$r_t$} &\multicolumn{1}{c}{} &\multicolumn{1}{c} {$\left\|e^n\right\|_{\infty}$} & \multicolumn{1}{c}{$r_t$}&\multicolumn{1}{c}{}  &\multicolumn{1}{c} {$\left\|e^n\right\|_{\infty}$} & \multicolumn{1}{c}{$r_t$}\\ \hline
    $1\slash10$&8.601e-04&1.02&&8.601e-04&1.02&&2.194e-03&1.09&&2.194e-03&1.09\\
   $1\slash20$&4.243e-04&1.01&&4.243e-04&1.01&&1.030e-03&1.06&&1.030e-03&1.06\\
   $1\slash40$&2.104e-04&1.01&&2.104e-04&1.01&&4.934e-04&1.04&&4.934e-04&1.04\\
  $1\slash80$&1.046e-04&&&1.046e-04&&&2.392e-04&&&2.392e-04&\\\hline
    \end{tabular}
    }
    \end{adjustwidth}
    \end{table}
\begin{table}[!ht]
  \caption{The errors, convergence orders in space, and CPU time for the fractional Fisher equation \eqref{fisher} with fixed temporal step size $\Delta t=2^{-14}$.} \label{tab4}
   \centering
    \begin{adjustwidth}{0cm}{0cm}
  \resizebox{!}{1.5cm}
  {
\begin{tabular}{cccccccccccc}
   \hline
 \multicolumn{1}{c}{\multirow{3}{0.5cm}{$h$}}& \multicolumn{5}{c}{$\alpha=0.2$} &\multicolumn{1}{c}{\multirow{1}{0cm}{}}  &\multicolumn{5}{c}{$\alpha=0.5$}\\
\cline{2-6}\cline{8-12}\multicolumn{1}{c}{} &\multicolumn{2}{c}{Fast scheme} &\multicolumn{1}{c}{\multirow{2}{0.1cm}{}}  &\multicolumn{2}{c}{Direct scheme}&\multicolumn{1}{c}{\multirow{2}{0cm}{}}&\multicolumn{2}{c}{Fast scheme} &\multicolumn{1}{c}{\multirow{2}{0cm}{}}  &\multicolumn{2}{c}{Direct scheme}\\
\cline{2-3}  \cline{5-6}\cline{8-9} \cline{11-12}\multicolumn{1}{c}{}&\multicolumn{1}{c} {$\left\|e^n\right\|_{\infty}$} & \multicolumn{1}{c}{$r_s$}&\multicolumn{1}{c}{}  &\multicolumn{1}{c} {$\left\|e^n\right\|_{\infty}$} & \multicolumn{1}{c}{$r_s$} &\multicolumn{1}{c}{} &\multicolumn{1}{c} {$\left\|e^n\right\|_{\infty}$} & \multicolumn{1}{c}{$r_s$}&\multicolumn{1}{c}{}  &\multicolumn{1}{c} {$\left\|e^n\right\|_{\infty}$} & \multicolumn{1}{c}{$r_s$}\\ \hline
  $1\slash80$&8.196e-04&1.96&&8.196e-04&1.96&&6.045e-04&1.97&&6.045e-04&1.97\\
  $1\slash160$&2.112e-04&2.01&&2.112e-04&2.01&&1.547e-04&2.01&&1.547e-04&2.01\\
  $1\slash320$&5.246e-05&2.00&&5.246e-05&2.00&&3.842e-05&2.00&&3.842e-05&2.00\\
  $1\slash640$&1.316e-05&&&1.316e-05&&&9.649e-06&&&9.649e-06&\\\hline
   CPU(s)&\multicolumn{2}{c}{38.80}&&\multicolumn{2}{c}{1071.91}&&\multicolumn{2}{c}{40.22}&&\multicolumn{2}{c}{850.58}\\\hline
    \end{tabular}
    }
    \end{adjustwidth}
    \end{table}

\begin{example}
\end{example} We consider the fractional Huxley equation
\begin{equation}\label{huxley}
  _0^C\!{D}_t^\alpha u(x,t)= u_{xx} -0.1u(1-u)(u-0.001),~\quad 0< t\leq 1
\end{equation}
with the double Gaussian initial value
$u(x,0)=\exp(-10(x-0.5)^2)+\exp(-10(x+0.5)^2)$. Tables \ref{tab5} and
\ref{tab6} present the numerical results for $\alpha = 0.2,\ 0.5$, which show
that our fast scheme (\ref{3.9f})-(\ref{3.11f}) has
the same convergence order ${O}(h^2+\Delta t)$ in $L_{\infty}$ norm as the direct scheme (\ref{3.8})-(\ref{3.11}), but takes much less computational time.

To demonstrate the complexity of the two schemes, we plot in Fig. \ref{fig2} the
CPU time in seconds for both schemes. We observe that our fast scheme has almost linear complexity in $N_T$ and is much faster than the direct scheme.
\begin{table}[!ht]
  \caption{The errors and convergence orders in time for the fraction Huxley equation \eqref{huxley} with fixed spatial step size $h=2^{-10}$.} \label{tab5}
   \centering
    \begin{adjustwidth}{0cm}{0cm}
  \resizebox{!}{1.36cm}
  {
\begin{tabular}{cccccccccccc}
   \hline
 \multicolumn{1}{c}{\multirow{3}{0.5cm}{$\Delta t$}}& \multicolumn{5}{c}{$\alpha=0.2$} &\multicolumn{1}{c}{\multirow{1}{0cm}{}}  &\multicolumn{5}{c}{$\alpha=0.5$}\\
 \cline{2-6}\cline{8-12}\multicolumn{1}{c}{} &\multicolumn{2}{c}
       {Fast scheme} &\multicolumn{1}{c}{\multirow{2}{0.1cm}{}}  &\multicolumn{2}{c}{Direct scheme}&\multicolumn{1}{c}{\multirow{2}{0cm}{}}&\multicolumn{2}{c}
       {Fast scheme} &\multicolumn{1}{c}{\multirow{2}{0cm}{}}  &\multicolumn{2}{c}
       {Direct scheme}\\
\cline{2-3}  \cline{5-6}\cline{8-9} \cline{11-12}\multicolumn{1}{c}{}&\multicolumn{1}{c} {$\left\|e^n\right\|_{\infty}$} & \multicolumn{1}{c}{$r_t$}&\multicolumn{1}{c}{}  &\multicolumn{1}{c} {$\left\|e^n\right\|_{\infty}$} & \multicolumn{1}{c}{$r_t$} &\multicolumn{1}{c}{} &\multicolumn{1}{c} {$\left\|e^n\right\|_{\infty}$} & \multicolumn{1}{c}{$r_t$}&\multicolumn{1}{c}{}  &\multicolumn{1}{c} {$\left\|e^n\right\|_{\infty}$} & \multicolumn{1}{c}{$r_t$}\\ \hline
  $1\slash10$&1.089e-03&1.04&&1.089e-03&1.04&&2.961e-03&1.06&&2.961e-03&1.06\\
  $1\slash20$&5.304e-04&1.02&&5.304-04&1.02&&1.418e-03&1.04&&1.418e-03&1.04\\
 $1\slash40$&2.614e-04&1.01&&2.614e-04&1.01&&6.896e-04&1.03&&6.896e-04&1.03\\
  $1\slash80$&1.294e-04&&&1.294e-04&&&3.379e-04&&&3.379e-04&\\\hline
    \end{tabular}
    }
    \end{adjustwidth}
\end{table}
\begin{table}[!ht]
  \caption{The errors, convergence orders in space, and CPU time for the fractional Huxley equation \eqref{huxley} with fixed temporal step size $\Delta t=2^{-14}$.} \label{tab6}
   \centering
    \begin{adjustwidth}{0cm}{0cm}
  \resizebox{!}{1.5cm}
  {
\begin{tabular}{cccccccccccc}
   \hline
 \multicolumn{1}{c}{\multirow{3}{0.5cm}{$h$}}& \multicolumn{5}{c}{$\alpha=0.2$} &\multicolumn{1}{c}{\multirow{1}{0cm}{}}  &\multicolumn{5}{c}{$\alpha=0.5$}\\
 \cline{2-6}\cline{8-12}\multicolumn{1}{c}{} &\multicolumn{2}{c}
       {Fast scheme} &\multicolumn{1}{c}{\multirow{2}{0.1cm}{}}  &\multicolumn{2}{c}{Direct scheme}&\multicolumn{1}{c}{\multirow{2}{0cm}{}}&\multicolumn{2}{c}
       {Fast scheme} &\multicolumn{1}{c}{\multirow{2}{0cm}{}}  &\multicolumn{2}{c}
       {Direct scheme}\\
\cline{2-3}  \cline{5-6}\cline{8-9} \cline{11-12}\multicolumn{1}{c}{}&\multicolumn{1}{c} {$\left\|e^n\right\|_{\infty}$} & \multicolumn{1}{c}{$r_s$}&\multicolumn{1}{c}{}  &\multicolumn{1}{c} {$\left\|e^n\right\|_{\infty}$} & \multicolumn{1}{c}{$r_s$} &\multicolumn{1}{c}{} &\multicolumn{1}{c} {$\left\|e^n\right\|_{\infty}$} & \multicolumn{1}{c}{$r_s$}&\multicolumn{1}{c}{}  &\multicolumn{1}{c} {$\left\|e^n\right\|_{\infty}$} & \multicolumn{1}{c}{$r_s$}\\ \hline
  $1\slash80$&9.051e-04&1.96&&9.051e-04&1.96&&6.882e-04&1.98&&6.882e-04&1.98\\
   $1\slash160$&2.319e-04&2.01&&2.319e-04&2.01&&1.746e-04&2.01&&1.746e-04&2.01\\
$1\slash320$&5.762e-05&2.00&&5.762e-05&2.00&&4.344e-05&2.00&&4.344e-05&2.00\\
  $1\slash640$&1.447e-05&&&1.447e-05&&&1.091e-05&&&1.091e-05&\\\hline
   CPU(s)&\multicolumn{2}{c}{38.69}&&\multicolumn{2}{c}{1106.58}&&\multicolumn{2}{c}{40.14}&&\multicolumn{2}{c}{855.87}\\\hline
    \end{tabular}
    }
    \end{adjustwidth}
\end{table}
\begin{figure}[htbp]
  \centering
  \includegraphics[width=0.45\textwidth]{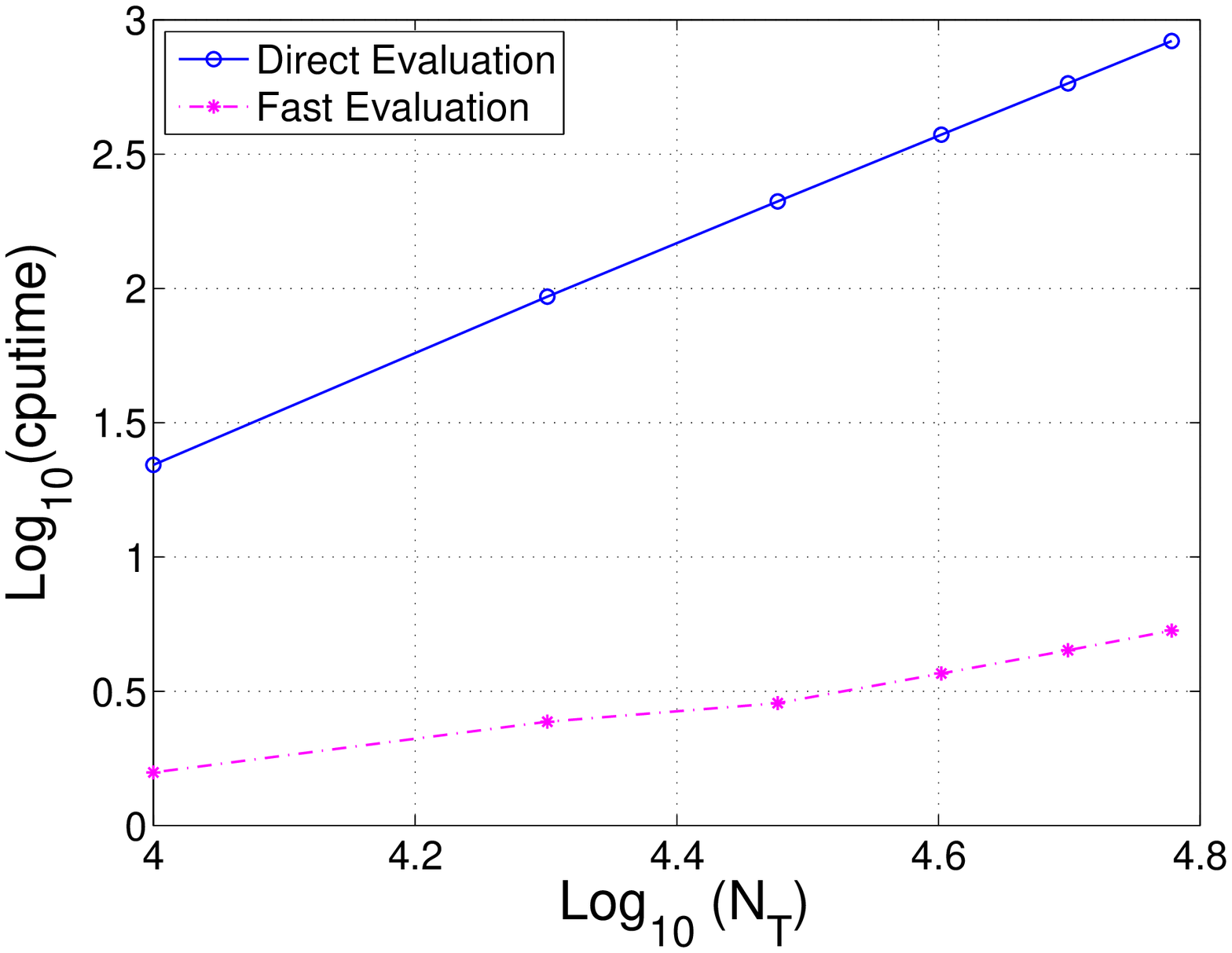}
   \includegraphics[width=0.45\textwidth]{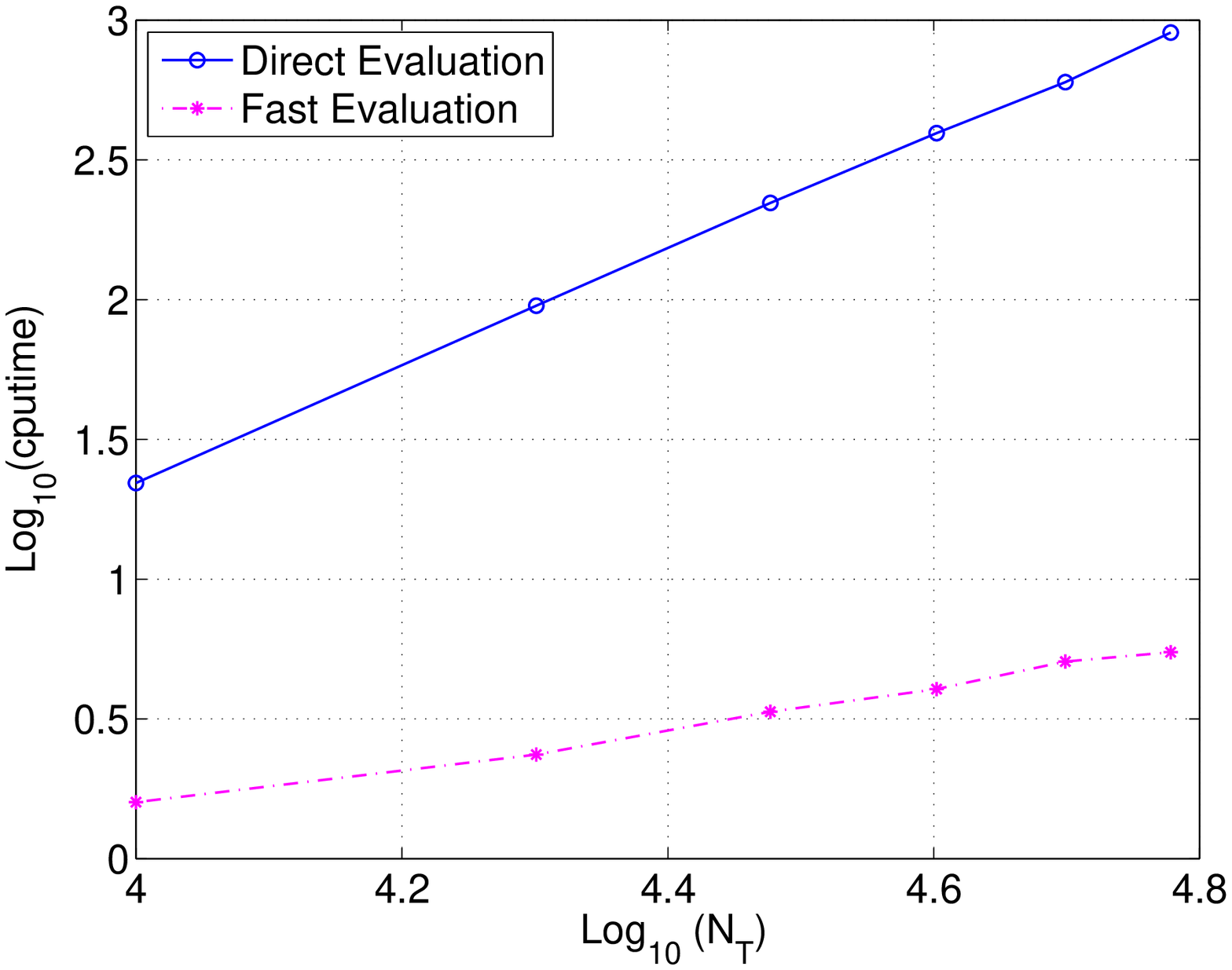}
  \caption{The log-log (in base $10$) plot of the CPU time (in seconds) versus the
    total number of time steps $N_T$ for two schemes. Here $N_S=80$ and
    $\alpha=0.5$. The left panel shows the results for the Fisher equation, and
    the right panel shows the results for the Huxley equation.}
    \label{fig2}
\end{figure}
\section{Conclusions}\label{Conclusions}
We have developed a fast algorithm for the evaluation of the Caputo
fractional derivative $_{0}^{C}\!D_{t}^{\alpha}f(t)$ for $0<\alpha<1$. The algorithm
relies on an efficient sum-of-exponentials approximation for the convolution
kernel $t^{-1-\alpha}$ with the absolute error $\varepsilon$ over the interval
$[\dt, T]$. Specifically, we have shown that the number of exponentials
needed
in the approximation is of the order $O\left(\log\frac{1}{\varepsilon}\left(
\log\log\frac{1}{\varepsilon}+\log\frac{T}{\dt}\right)+\right.$
$\left.\log\frac{1}{\dt}\left(
\log\log\frac{1}{\varepsilon}+\log\frac{1}{\dt}\right)
\right)$, which removes the term $O\left(\log^2\frac{1}{\varepsilon}\right)$
in \cite{beylkin3,Li10SISC}. The resulting algorithm has nearly optimal
complexity in both CPU time and storage.

We then applied our fast evaluation scheme of the Caputo derivative
to solve the fractional diffusion equations. We first demonstrated that it is
straightforward to incorporate our fast algorithm into the existing finite
difference schemes for solving the fractional diffusion equations. We
then proved a prior estimate about the solution of our new FD scheme which
leads to the stability of the new scheme. We also presented a rigorous error
bound for the new scheme. Finally, the numerical results on linear and nonlinear
fraction diffusion equations show that our new scheme has the same order of
convergence as the existing standard FD schemes, but with nearly optimal
complexity in CPU time and storage.

Our work can be extended along several directions. First, it is straightforward
to design high order schemes for the evaluation of fractional derivatives.
Second, one may develop fast high-order algorithms for solving fractional PDEs which contains
fractional derivatives in both time and space when the current scheme is combined
with other existing schemes \cite{ChenMC12,CuiJCP09,CuiJCP12,jiang4}.
Third, efficient and stable artificial boundary conditions can be designed using similar techniques
in \cite{jiang3} for solving
fractional PDEs in high dimensions. These issues are currently under investigation and the results
will be reported on a later date.
%
%
%
%
%
%
\appendix
\section{The Proof of Lemma \ref{lemma3}}
\begin{proof}
  Applying the definition \eqref{defF} of the fast evaluation scheme and the
  Cauchy-Schwarz inequality, we have
\begin{align}\label{a1}
  ({^{FC}_{0}}\mathbb{D}_t^{\alpha}g^{k})g^k&=\frac{1}{\Delta t^{\alpha}\Gamma(1-\alpha)}\bigg(\frac{1}{1-\alpha}(g^k)^2-(\frac{\alpha}{1-\alpha}+a_0)g^{k-1}g^k\nonumber\\
&\quad  -\sum_{l=1}^{k-2}(a_{k-l-1}+b_{k-l-2})g^lg^k
-(b_{k-2}+\frac{1}{2k^{\alpha}})g^0g^k\bigg)\nonumber\\
  &\geq\frac{1}{\Delta t^{\alpha}\Gamma(1-\alpha)}\left[\bigg(\frac{2-\alpha}{2(1-\alpha)}-\frac{1}{2}\sum_{l=0}^{k-2}(a_{l}+b_{l})-\frac{1}{k^{\alpha}}\bigg)(g^{k})^2
    \right.\nonumber\\
    &\quad \left.-\frac{1}{2}(\frac{\alpha}{1-\alpha}+a_0)(g^{k-1})^2
    -\frac{1}{2}\sum_{l=1}^{k-2}(a_{k-l-1}+b_{k-l-2})(g^l)^2\right.\nonumber\\
   &\quad\left. -\frac{1}{2}(b_{k-2}+\frac{1}{k^{\alpha}})(g^0)^2\bigg)\right].
\end{align}
Summing the above inequality from $k=1$ to $n$, we obtain
\begin{align}\label{a2}
&\Delta t\sum_{k=1}^n({^{FC}_{0}}\mathbb{D}_t^{\alpha}g^{k})g^k
\geq\frac{\Delta t^{1-\alpha}}{\Gamma(1-\alpha)}\left[\sum_{k=2}^n\bigg(\frac{1}{1-\alpha}-\frac{\alpha}{2(1-\alpha)}-\frac{1}{2}\sum_{l=0}^{k-2}(a_{l}+b_{l})-\frac{1}{2k^{\alpha}}\bigg)(g^{k})^2\right.\nonumber\\
&-\left.\sum_{k=2}^n\bigg(\frac{\alpha}{2(1-\alpha)}+\frac{a_0}{2}\bigg)(g^{k-1})^2-\frac{1}{2}\sum_{k=2}^n\sum_{l=1}^{k-2}\bigg(a_{k-l-1}+b_{k-l-2}\bigg)(g^l)^2\right.\nonumber\\
&+\left.\frac{1}{2(1-\alpha)}(g^1)^2-\frac{1}{2(1-\alpha)}(g^0)^2-\frac{1}{2}\sum_{k=2}^n\bigg(b_{k-2}+\frac{1}{k^{\alpha}}\bigg)(g^0)^2\right]\nonumber\\
&= \frac{\Delta t^{1-\alpha}}{\Gamma(1-\alpha)}\sum_{k=1}^n\bigg(C_k(g^k)^2-C_0(g^0)^2\bigg),
\end{align}
%
%
%
%
%
%
where the coefficients $C_k$ ($k=0,1,\cdots,n$) are given by the formula
\begin{align}\label{a3}
C_k =\left \{ \begin{array}{ll}
  \frac{1}{2(1-\alpha)}+\frac{1}{2}\sum\limits_{k=2}^n (b_{k-2}+\frac{1}{k^\alpha}), & \; k=0,\\
  \frac{1}{2}-\frac{1}{2}\sum\limits_{l=0}^{n-2}(a_l+b_{l}) +\frac{1}{2}b_{n-2}, & \;  k = 1,\\
  1-\frac{1}{2}\sum\limits_{l=0}^{k-2}(a_l+b_l)-\frac{1}{2k^\alpha}-\frac{1}{2}\sum\limits_{l=0}^{n-k-1}(a_{l}+b_{l})+\frac{1}{2}b_{n-k-1},  & \; 2\leq k< n,\\
\frac{2-\alpha}{2(1-\alpha)}-\frac{1}{2}\sum\limits_{l=0}^{k-2}(a_l+b_l)-\frac{1}{2k^\alpha},& \; k=n.
\end{array} \right.
\end{align}
From \eqref{2.2.1}, we have the estimate
\begin{equation}\label{a4}
\frac{1}{t^{1+\alpha}}-\varepsilon\leq\sum_{j=1}^{\nexp}\omega_j{e^{-s_jt}}\leq\frac{1}{t^{1+\alpha}}+\varepsilon.
\end{equation}
It is also straightforward to verify that
\begin{equation}\label{a5}
  \sum_{l=0}^{k-2}(a_l+b_l)
  =\alpha\Delta t^{\alpha}\int_{\Delta t}^{k\Delta t}
  \sum_{j=1}^{\nexp}\omega_j{e^{-s_jt}}\mathrm{d}t.
\end{equation}
Combining \eqref{a4} and \eqref{a5}, we obtain
\begin{equation}\label{a6}
 (1-\frac{1}{k^\alpha})-\alpha\Delta t^\alpha t_{k-1}\varepsilon\leq\sum_{l=0}^{k-2}(a_l+b_l)\leq(1-\frac{1}{k^\alpha})+\alpha\Delta t^\alpha t_{k-1}\varepsilon,
\end{equation}
%
%
%
%
%
%
Substituting \eqref{a6} into \eqref{a3} yields the following estimates

\begin{align}\label{a7}
\left \{ \begin{array}{ll}
C_0&\leq\frac{n^{1-\alpha}}{(1-\alpha)}+\alpha\Delta t^\alpha t_{n-1}\varepsilon,\\
C_1&\geq\frac{1}{2}-\frac{1}{2}\sum\limits_{l=0}^{n-2}(a_l+b_l)\geq\frac{1}{2n^{\alpha}}-\alpha\Delta t^\alpha t_{n-1}\varepsilon,\\
C_k&=1-\frac{1}{2}\sum\limits_{l=0}^{k-2}(a_l+b_l)-\frac{1}{2k^\alpha}-\frac{1}{2}\sum\limits_{l=0}^{n-k-1}(a_{l}+b_{l})+\frac{1}{2}b_{n-k-1}\\
&\geq\frac{1}{2n^{\alpha}}-\alpha\Delta t^\alpha t_{n-1}\varepsilon,\quad 2\leq k\leq n-1,\\
C_n&\geq\frac{2-\alpha}{2(1-\alpha)}-\sum\limits_{l=0}^{n-2}(a_l+b_l)-\frac{1}{2n^\alpha}\geq\frac{1}{2n^{\alpha}}-\alpha\Delta t^\alpha t_{n-1}\varepsilon.\\
\end{array} \right.
\end{align}
Combining \eqref{a2} and \eqref{a7}, we obtain Lemma \ref{lemma3}.
\end{proof}
\end{document}